\documentclass{article}
\usepackage[margin=1.00in]{geometry}
\usepackage{graphicx} 
\usepackage{subcaption}
\usepackage{amsmath}
\usepackage{amsthm}
\usepackage{amssymb}
\usepackage{mathtools} 
\usepackage{yhmath} 
\usepackage{stmaryrd}
\usepackage{xcolor}
\usepackage{relsize}
\usepackage{authblk} 

\usepackage{algorithm}
\usepackage{algorithmicx}
\usepackage[pagewise]{lineno}
\usepackage{url}

\newtheorem{proposition}{Proposition}
\newtheorem{prop}{Property}
\newtheorem{rem}{Remark}

\title{A Nodal Discontinuous Galerkin Method with Rank-Adaptive Velocity Space Representation for the Multiscale BGK Model}

\author[1]{Andr\'es Galindo-Olarte}
\author[2]{Joseph Nakao}
\author[3]{Mirjeta Pasha}
\author[4]{Jing-Mei Qiu}
\author[5]{William Taitano}

\affil[1]{Oden Institute for Computational Engineering and Sciences, The University of Texas at Austin, Austin, TX, USA\newline
Email: afgalindo@utexas.edu, ORCID iD: 0009-0007-1551-0889}
\affil[2]{Department of Mathematics and Statistics, Swarthmore College, Swarthmore, PA, USA\newline
Email: jnakao1@swarthmore.edu, ORCID iD: 0009-0008-6589-4013}
\affil[3]{Department of Mathematics, Virginia Polytechnic and State University, Blacksburg, VA, USA\newline
Email: mpasha@vt.edu, ORCID iD: 0000-0003-4249-2421}
\affil[4]{Department of Mathematical Sciences, University of Delaware, Newark, DE, USA\newline
Email: jingqiu@udel.edu, ORCID iD: 0000-0002-3462-188X}
\affil[5]{Theoretical Division, Los Alamos National Laboratory, Los Alamos, NM, USA\newline
Email: taitano@lanl.gov, ORCID iD: 0000-0002-2369-0935}
\date{}

\begin{document}
\maketitle
\let\thefootnote\relax\footnotetext{\hspace*{-1.8em}\textbf{Corresponding author:} Joseph Nakao
\textbf{Funding:} This material is based upon work supported by the National Science Foundation under Grant No. DMS-1929284 while the authors were in residence at the Institute for Computational and Experimental Research in Mathematics in Providence, RI, during the \emph{Empowering a Diverse Computational Mathematics Research Community} program. JQ gratefully acknowledges support from AFOSR grants FA9550-24-1-0254 and DOE grant DE-SC0023164. MP gratefully acknowledges support from NSF DMS-2410699. Any opinions, findings, conclusions, or recommendations expressed in this material are those of the authors and do not necessarily reflect the views of the National Science Foundation. WT was supported by Triad National Security, LLC under contract 89233218CNA000001, and the DOE Office of Applied Scientific Computing Research (ASCR) through the Mathematical Multifaceted Integrated Capability Centers program.}

\begin{abstract}
\noindent A novel hybrid algorithm is presented for the Boltzmann-BGK equation, in which a rank-adaptive decomposition is applied solely in the velocity subspace, while a full-rank representation is maintained in the physical (position) space. This approach establishes a foundation for extending modern rank-adaptive techniques to solve the Boltzmann equation in realistic settings, particularly where structured representations, such as conformal geometries, may not be feasible in practical engineering applications. A nodal discontinuous Galerkin method is employed for spatial discretization, coupled with a rank-adaptive decomposition over the velocity grid, as well as implicit-explicit Runge-Kutta methods for time integration. To handle the limit of vanishing collision time, a multiscale implicit integrator based on an auxiliary moment equation is utilized. The algorithm’s order of accuracy, reduced computational complexity, and robustness are demonstrated on a suite of canonical gas kinetics problems with increasing complexity.
\end{abstract}


\section{Introduction}
\label{sec:introduction}

The Boltzmann transport equation governs the dynamical evolution of the particle distribution function (PDF), $f\left(t, {\bf x}, {\bf v}\right)$, defined on the position ${\bf x} \in \mathbb{R}^3$ and velocity ${\bf v} \in \mathbb{R}^3$ phase space over time $t \in \mathbb{R}_+$. The equation mediates collective many-body interactions through the collision operator, driving the local PDF toward equilibrium (e.g., a Maxwellian) on a characteristic timescale $\varepsilon$. It provides a statistical bridge between microscopic (e.g., molecular dynamics) and macroscopic (e.g., Euler or Navier-Stokes) descriptions, and is therefore indispensable in modeling systems where macroscopic fidelity is insufficient and microscopic simulations are computationally prohibitive. Applications span a wide range, including rarefied gas dynamics for high-altitude reentry vehicles, photon and neutron transport in high-energy-density and astrophysical systems, and plasma modeling for semiconductor etching and thermonuclear fusion devices.

Numerical modeling of the Boltzmann equation faces significant challenges. The high-dimensional nature of the PDF leads to the \textit{curse of dimensionality}, where computational complexity scales as ${\cal O}(N^6)$, with $N$ being the number of grid points per dimension (3 dimensions in physical space, 3 dimensions in velocity space). In addition, sophisticated numerical techniques are required to satisfy the equation's numerous constraints—often integral or differential in nature—including global conservation of mass, momentum, and energy, compliance with entropy maximization (as mandated by the Boltzmann ${\cal H}$-theorem), and positivity of the distribution function. The system is also inherently multiscale, with large separations between the fast collision timescale and the slower macroscopic fluid timescales. Thus, numerical schemes must often step over the fast yet dynamically insignificant collision timescale while ensuring consistency with the asymptotic fluid limit, without resorting to costly inversion of dense integral operators.

A substantial body of literature has addressed these challenges. While a comprehensive review is beyond the scope of this work, we briefly highlight relevant approaches. Asymptotic preserving (AP) and conservative solvers have been developed, such as the micro-macro decomposition (MMD) framework in Refs. \cite{bennoune2008uniformly, gamba_jcp_mmd_2019,liu_cmp_mmd_2004, lemou_jsc_2008_mmd}, which evolves both the fluid subsystem and a perturbation to the PDF, thereby preserving conservation and asymptotic behavior in the $\varepsilon \to 0$ limit within the fluid subsystem and picking up the kinetic corrections to the viscous stress and heat flux through the perturbative component. Other methods like the high-order low-order (HOLO) method maintain a full PDF representation but introduce \textit{auxiliary} moment equations in conservative form to evaluate the collision operator more efficiently and alleviate stiffness due to small $\varepsilon$ \cite{chacon_jcp_2017_holo_review, taitano_jctt_2014}. Modern approaches have also explored low-rank tensor decomposition techniques \cite{einkemmer2025review} to combat the curse of dimensionality by factorizing the PDF into products of lower-dimensional tensors, yielding linear scaling with dimension. However, these techniques typically require structured $d$-dimensional hypercube representations of the entire PDF, which are often incompatible with unstructured data ubiquitous in realistic engineering domains.

Discontinuous Galerkin (DG) methods are a popular choice for solving kinetic models and convection-dominated problems due to their many favorable properties \cite{cockburn2000development,cockburn2001runge,hesthaven2008nodal}; they are stable, arbitrarily high-order accurate in space, $h$-$p$ adaptive, highly parallelizable, and flexible for working with complicated geometries. Furthermore, DG methods use piecewise polynomial spaces defined over the elements, which when coupled with high resolution methods \cite{shu2009high,toro2013riemann}, are better suited to deal with discontinuous solutions. Several DG methods have been developed to solve the BGK equation \cite{alekseenko2012application,alexeenko2008high,ding2021semi,xiong2015high}, a simplified relaxation model for the Boltzmann equation in the context of the kinetic description of gases originally introduced in \cite{bhatnagar1954model}; this is also known as the Boltzmann-BGK model. The methods developed in \cite{ding2021semi,xiong2015high} leverage nodal DG methods \cite{hesthaven2008nodal}, a class of DG methods that uses a Lagrangian polynomial basis defined by Gaussian quadrature nodes. Doing so simplifies to solving for the nodal values of the solution, which can be advantageous when working with kinetic solutions.

Many other methods have been developed to address the curse of dimensionality, structure-preservation, and asymptotic-preservation inherent to kinetic simulations. Multiscale numerical schemes for the BGK equation that satisfy the AP property include the works in \cite{ding2021semi,hu2018asymptotic,hu2017class}. To address high-dimensionality, several rank-adaptive and low-rank methods have been presented that exploit the underlying low-rank structures commonly enjoyed in kinetic simulations. By using rank-adaptive decompositions of the PDF, we dramatically reduce the storage and computational complexities by avoiding the storage of the entire $d$-dimensional array. Two classes of low-rank methods have seen particular attention: dynamical low-rank (DLR), and step-and-truncate (SAT) methods; we refer the reader to the thorough review paper on low-rank methods for kinetic simulations \cite{einkemmer2025review}. In the context of the BGK equation, low-rank methods include: DLR approaches \cite{baumann2024stable,dektor2024interpolatory,einkemmer2021efficient,wang2025dynamical}, a conservative sampling based adaptive rank approach \cite{sands2025adaptive}, and low-rank frameworks are utilized with HOLO and micro-macro methods \cite{hauck2024high}. Further numerical schemes combining these frameworks have been proposed for kinetic problems: asymptotic-preserving DLR methods have been developed in \cite{einkemmer2024asymptotic,einkemmer2021asymptotic,einkemmer2025asymptotic}, and conservative SAT methods were proposed in \cite{guo2024local,guo2022low,guo2024conservative}. Also worth noting is the conservative and positivity-preserving method for the multi-species BGK model in \cite{haack2023numerical}.

To address the challenges of solving the Boltzmann-BGK equation, we present a unified framework that simultaneously resolves the key issues of: 1) conservation, 2) geometric complexity and high-order accuracy, and 3) the curse of dimensionality. Numerical experiments also show that we reasonably capture the multiscale nature of the equation. Our proposed method is a \textit{hybrid} formulation for solving the Boltzmann-BGK equation, wherein rank-adaptive factorization is applied exclusively in the velocity subspace, while a full-rank structure is retained in the physical (position) space. In many practical systems, ${\bf v} \in \mathbb{R}^3$, but ${\bf x} \in \Omega_x \subset \mathbb{R}^3$, where $\Omega_x$ may contain non-conformal boundaries, holes, or other geometric complexities. These domains can often be partitioned into general polyhedral elements, $\Omega_x = \bigcup_i {\cal I}_i$, with ${\cal I}_i$ denoting an arbitrarily shaped $i^{\text{th}}$ element. Such features motivate the use of a hybrid representation. We develop an accompanying numerical scheme based on a nodal DG discretization in physical space, coupled with a uniform rank-adaptive discretization in velocity space. By doing so, the geometrical complexity encountered in the physical space can be more naturally handled while still achieving complexity reduction in the velocity space, granted by rank-adaptive factorization.

In this work, we demonstrate the proof-of-principle of the approach on a 1d2v Boltzmann-BGK equation and develop the hybrid rank-adaptive nodal DG framework that will serve as a basis for a future extension to a multidimensional setting with complex geometries. The scheme also incorporates a stiff, multiscale time integrator (e.g., implicit-explicit Runge-Kutta methods \cite{ascher1997implicit}) that leverages intermediate auxiliary fluid moment equations to permit large time steps. Furthermore, in the asymptotic limit $\varepsilon \rightarrow 0$, the proposed algorithm converges to the local Maxwellian distribution. This behavior is achieved through the exact discrete conservation of mass, momentum, and energy, facilitated by a novel adaptation of the quadrature-corrected moments (QCM) technique \cite{taitano_jcp_2017_ep_rfp}, which ensures the vanishing of the discrete moments of the collision operator. To control spurious oscillations near sharp gradients, we further employ a post-processing slope limiting procedure to control high-order polynomial modes.

The remainder of the manuscript is organized as follows: Section \ref{sec: model} introduces the model Boltzmann-BGK equation in reduced 1d2v form. Section \ref{sec: algorithm} details the numerical scheme, covering the nodal DG discretization in physical space, the auxiliary moment equations, the rank-adaptive velocity representation with adaptive rank procedures, and the IMEX time integration strategy. Section \ref{sec: numerics} presents numerical tests that demonstrate the accuracy, computational complexity, and robustness of the proposed scheme. Concluding remarks are given in Section \ref{sec: conclusion}.
\section{The model of interest}\label{sec: model}
We consider the non-dimensionalized BGK model
\begin{equation}\label{eq: BGK}
    \partial_tf + \mathbf{v}\cdot\nabla_{\mathbf{x}}f = \frac{1}{\varepsilon}(M_{\mathbf{U}}-f),
\end{equation}
where $f=f(\mathbf{x},\mathbf{v},t)$ describes the distribution of particles in phase-space $(\mathbf{x},\mathbf{v})\in\Omega_x\times\Omega_v$ and time $t>0$, $f_0(\mathbf{x},\mathbf{v})$ describes the initial distribution, and appropriate boundary conditions are prescribed. The dimensionless parameter $\varepsilon>0$ is the Knudsen number proportional to the mean-free path and can be spatially dependent. The BGK model is multiscale in nature since $0<\varepsilon\ll 1$ describes high collisionality (limiting fluid regime), and $\varepsilon=\mathcal{O}(1)$ describes lower collisionality (kinetic regime). The local Maxwellian $M_{\mathbf{U}}=M_{\mathbf{U}}(\mathbf{x},\mathbf{v},t)$ is defined by the macroscopic quantities,
\begin{equation}\label{eq: MU}
    M_{\mathbf{U}}(\mathbf{x},\mathbf{v},t) = \frac{n(\mathbf{x},t)}{(2\pi T(\mathbf{x},t))^{d/2}}\text{exp}\left(-\frac{|\mathbf{v}-\mathbf{u}(\mathbf{x},t)|^2}{2T(\mathbf{x},t)}\right),
\end{equation}
where $d \geq 1$ is the number of velocity dimensions ($\Omega_v\subset\mathbb{R}^d$), and $n$, $\mathbf{u}$, and $T$ are the number density, mean/bulk velocity, and temperature, respectively. These macroscopic quantities are obtained by taking the first few moments of the distribution function,
\begin{equation}\label{eq: U}
    \mathbf{U}(\mathbf{x},t)\coloneqq \left(n,n\mathbf{u},E\right)^T = \int_{\mathbb{R}^d}{\left(1,\mathbf{v},\frac{|\mathbf{v}|^2}{2}\right)^Tfdv},
\end{equation}
where $E(\mathbf{x},t)=(n|\mathbf{u}|^2+dnT)/2$ is the total energy, which is related to the temperature
\begin{equation}
    T(\mathbf{x},t)=\frac{1}{dn}\int_{\mathbb{R}^d}{|\mathbf{v}-\mathbf{u}|^2fdv}.
\end{equation}

The BGK equation is a relaxation model for the Boltzmann equation. As such, the BGK collision operator $Q(f)\coloneqq M_{\mathbf{U}}-f$ satisfies many well-established properties. We state two such properties here (summarized in \cite{ding2023accuracy}), and refer the reader to \cite{cercignani1988boltzmann,cercignani2000rarefied,cercignani1969mathematical} for further details and properties of the BGK operator.

\begin{prop}\label{prop: collisional_invariants}
There exist three collisional invariants for the BGK operator. In particular,
\begin{equation}
    \int_{\mathbb{R}^d}{Q(f)\phi_idv}=\int_{\mathbb{R}^d}{(M_{\mathbf{U}}-f)\phi_idv} = 0,\qquad i=1,2,3,
\end{equation}
where $(\phi_1,\phi_2,\phi_3)\coloneqq(1,\mathbf{v},|\mathbf{v}|^2/2)$. That is, the BGK operator conserves mass, momentum, and energy.
\end{prop}

\begin{prop}\label{prop: equilibrium}
In the limit as $\varepsilon\rightarrow 0$,
\begin{equation}
    Q(f)= M_{\mathbf{U}}-f=0\qquad\Longleftrightarrow\qquad f=M_{\mathbf{U}}[f]=M_{\mathbf{U}}[{\mathbf{U}}],
\end{equation}
where ${\mathbf{U}}$ are the moments of $f$, and $M_{\mathbf{U}}$ only depends on $f$ through its moments. Hence, along the spatial characteristics defined by $\frac{d\mathbf{x}}{dt}=\mathbf{v}$ (and so $\frac{df}{dt}=\frac{Q(f)}{\varepsilon}$), in equilibrium $f$ is the local Maxwellian distribution defined by its moments.
\end{prop}


\section{The numerical scheme}\label{sec: algorithm}

In this section, we adopt an asymptotic-preserving framework for the BGK equation (see \cite{filbet2010class}), while leveraging and combining it with other efficient methods in position and velocity spaces. Similar to \cite{xiong2015high}, we discretize in space using a nodal DG (NDG) framework. Whereas, we take advantage of the low-rank structure that kinetic solutions are known to exhibit, opting to evolve a low-rank decomposition of the solution in velocity space. Low-rank methods significantly reduce the storage complexity, and hence the computational complexity. By coupling NDG and low-rank methodologies, we propose a highly efficient framework for solving the BGK equation. Lastly, we integrate in time using a stiffly accurate implicit-explicit (IMEX) Runge-Kutta method that includes a diagonally implicit Runge-Kutta (DIRK) method for evolving the stiff collision operator, especially for small $\varepsilon$. We only consider a second-order IMEX RK method, but extension to higher-order methods is straightforward.

\subsection{Nodal DG discretization in physical space}
Assuming one spatial dimension and two velocity dimensions, the BGK equation simplifies to
\begin{equation}\label{eq: BGK_1d2v}
    \frac{\partial f}{\partial t} + v_x\frac{\partial f}{\partial x} = \frac{1}{\varepsilon}(M_{\mathbf{U}}-f),
\end{equation}
where $f=f(x,v_x,v_y,t)$, with $\Omega_x=[a,b]$ and $\Omega_v=[-V_{\text{max}},V_{\text{max}}]^2$. The Knudsen number $\varepsilon(x)$ is assumed to be spatially dependent. We discretize the spatial domain into $N_x$ intervals,
\[a=x_{\frac{1}{2}}<x_{\frac{3}{2}}<...<x_{N_x-\frac{1}{2}}<x_{N_x+\frac{1}{2}}=b,\]
generating a mesh consisting of elements $I_i\coloneqq[x_{i-\frac{1}{2}},x_{i+\frac{1}{2}}]$, $i=1,2,...,N_x$. For simplicity, we assume a uniform mesh with $h_x=x_{i+\frac{1}{2}}-x_{i-\frac{1}{2}}$ for all $i$. Define the finite dimensional discrete space
\begin{equation}
    V_h^k\coloneqq \{g\in L^2(\Omega_x)\ :\  g|_{I_i}\in \mathbb{P}^k(I_i), 1\leq i\leq N_x\},
\end{equation}
where $k\geq 0$ is a non-negative integer, and $\mathbb{P}^k(I_i)$ denotes the space of polynomials of degree at most $k$ on the element $I_i$. For each element $I_i$ ($1\leq i\leq N_x$), we consider the associated polynomial basis
\begin{equation}
    \{\phi^i_q\in \mathbb{P}^k(I_i)\ :\  1\leq q\leq k+1\},
\end{equation}
where the basis functions $\phi_q^i$ are orthogonal with respect to the $L^2$ inner product over $I_i$, which we denote by $\langle\cdot,\cdot\rangle_{I_i}$. For now, we only focus on discretizing in space by first restricting the solution to an element $I_i$, which we denote $f^i(x\in I_i,v_x,v_y,t) = f(x\in I_i,v_x,v_y,t)$. We desire the weak solution $f_h(\cdot,v_x,v_y,t)\in V_h^k$ such that for all $i=1,2,...,N_x$ and for all $\varphi\in V_h^k$,
\begin{equation}\label{eq: weakform}
    \int_{I_i}{\frac{\partial f_h}{\partial t}\varphi dx} + \widehat{(v_{x}f_h)}_{i+\frac{1}{2}}\varphi^-_{i+\frac{1}{2}} - \widehat{(v_{x}f_h)}_{i-\frac{1}{2}}\varphi^+_{i-\frac{1}{2}} - \int_{I_i}{v_{x}f_h\frac{d\varphi}{dx}dx} = \int_{I_i}\frac{1}{\varepsilon}\left(M_{\mathbf{U}}-f_h\right)\varphi dx,
\end{equation}
where subscript $i\pm\frac{1}{2}$ indicates an approximation as $x\rightarrow x_{i\pm\frac{1}{2}}$, superscript $\pm$ indicates the right or left limit, and the upwind numerical fluxes for $[v_{x}f_h]_{\partial I_i}$ are 
\begin{equation}
    \widehat{v_{x}f_h}=\begin{cases}
        v_{x}f_h^-,&\text{ if }v_{x}\geq 0,\\
        v_{x}f_h^+,&\text{ if }v_{x}<0.
    \end{cases}
\end{equation}

For each element $I_i$ ($1\leq i\leq N_x$), we have that $f_h|_{I_i}\approx f^i$.  Decomposing the weak solution (restricted to an element $I_i$) using the associated polynomial basis,
\begin{equation}\label{eq: fimn}
    f_h(x\in I_i,v_x,v_y,t) = \sum\limits_{q=1}^{k+1}{C^{i}_q(v_x,v_y,t)\phi_q^i(x)}.
\end{equation}

Plugging \eqref{eq: fimn} into equation \eqref{eq: weakform} and integrating against each basis function $\phi^i_p\in V_h^k$, we obtain the differential equation for the coefficients,
\begin{equation}
    m^i_{pp}\frac{dC^{i}_p}{dt} + \sum\limits_{q=1}^{k+1}{\left(\Big[\widehat{v_{x}C^{i}_{q}\phi^i_q}\phi^i_p\Big]_{\partial I_i} - v_{x}C^{i}_{q}b^i_{pq}\right)} = \Big\langle \frac{1}{\varepsilon}(M_{\mathbf{U}}-f_h),\phi^i_p\Big\rangle_{I_i},
\end{equation}
where the entries of the (diagonal) mass matrix $\mathbf{M}^i\in\mathbb{R}^{(k+1)\times(k+1)}$ and stiffness matrix $\mathbf{B}^i\in\mathbb{R}^{(k+1)\times(k+1)}$ are respectively
\begin{equation}
    m^i_{pq} = \Big\langle\phi^i_q,\phi^i_p\Big\rangle_{I_i},\qquad\qquad b^i_{pq} = \Big\langle\phi^i_q,\frac{d\phi^i_p}{dx}\Big\rangle_{I_i}.
\end{equation}

Using an upwind discretization to evaluate the flux term at the boundaries of cell $I_i$,
\begin{align}\label{eq: Cipmn}
\begin{split}
    m^i_{pp}\frac{dC^{i}_{p}}{dt} &+ \max{(v_{x},0)}\sum\limits_{q=1}^{k+1}{\left( C^{i}_{q}\phi^i_q(x^-_{i+\frac{1}{2}})\phi^i_p(x^-_{i+\frac{1}{2}}) - C^{i-1}_{q}\phi^{i-1}_q(x^-_{i-\frac{1}{2}})\phi^i_p(x^+_{i-\frac{1}{2}}) - C^{i}_{q}b^i_{pq}\right)}\\
    &+ \min{(v_{x},0)}\sum\limits_{q=1}^{k+1}{\left( C^{i+1}_{q}\phi^{i+1}_q(x^+_{i+\frac{1}{2}})\phi^i_p(x^-_{i+\frac{1}{2}}) - C^{i}_{q}\phi^{i}_q(x^+_{i-\frac{1}{2}})\phi^i_p(x^+_{i-\frac{1}{2}}) - C^{i}_{q}b^i_{pq}\right)}\\
    &= \Big\langle \frac{1}{\varepsilon}(M_{\mathbf{U}}-f_h),\phi^i_p\Big\rangle_{I_i},
\end{split}
\end{align}
where $\phi^i_q(x^-_{i+\frac{1}{2}})$ and $\phi^{i+1}_q(x^+_{i+\frac{1}{2}})$ respectively denote the left and right limits as $x\rightarrow x_{i+\frac{1}{2}}$. Note that upwinding only applies to the solution, not the test function $\phi^i_p$ which was restricted to $x\in I_i$.

Up to this point, we have followed a more general DG framework that allows any suitable (orthogonal) basis. Now, as with other nodal DG methods, we choose a local nodal basis consisting of Lagrange polynomials. In particular, we let $\phi^i_q(x)=L^i_q(x)\in\mathbb{P}^k(I_i)$ for $q=1,2,...,k+1$, where
\begin{equation}
    L^i_q(x)=\prod\limits_{\substack{1\leq s\leq k+1\\s\neq q}}{\frac{x-x^i_s}{x^i_p-x^i_s}}
\end{equation}
is the Lagrangian defined over $I_i$ and associated with the $k+1$ Gauss-Legendre quadrature nodes $\{x^i_s\in I_i\ :\ s=1,...,k+1\}$. Mapping from element $I_i$ onto the reference element $I\coloneqq[-\frac{1}{2},\frac{1}{2}]$ using a linear transformation, the quadrature nodes on $I$ are $\xi_s=(x^i_s-x_i)/h_x$, $s=1,...,k+1$. Further let $\{w_s\ :\ s=1,...,k+1\}$ denote the corresponding quadrature weights on $I$, and let $L_q(x)$ for $q=1,...,k+1$ denote the associated Lagrange polynomials on $I$. The nodal/Lagrangian basis simplifies many of the operations in equation \eqref{eq: Cipmn}. Mapping to from $I_i$ to the reference element $I$, one can easily verify the following:
\begin{equation}
    m^i_{pp} = h_xw_p,
\end{equation}
\begin{equation}
    b^i_{pq} = w_q\frac{dL_p}{d\xi}(\xi_q),
\end{equation}
\begin{equation}
    \Big\langle \frac{1}{\varepsilon}(M_{\mathbf{U}}-f_h),L^i_p\Big\rangle_{I_i} = \frac{h_xw_p}{\varepsilon^i_p}\left(M_{\mathbf{U}}(x^i_p,v_x,v_y,t)-C^i_p\right),
\end{equation}
where $\varepsilon(x^i_p)=\varepsilon^i_p$. Another advantage to using the Lagrangian nodal basis is that the coefficients are identical to nodal values of the solution, that is, $f^i(x^i_p,v_x,v_y,t)=C^{i}_{p}(v_x,v_y,t)$ for all $i=1,...,N_x$ and all $p=1,...,k+1$, by \eqref{eq: fimn}. That is, computing the coefficients in turn produces the nodal values of the solution.

In velocity space, we discretize $\Omega_v$ using a tensor product of uniform computational grids, $\mathbf{v}_x\otimes\mathbf{v}_y$, with
\[\mathbf{v}_x:v_{x,1}<...<v_{x,m}<...<v_{x,N_{v_x}},\qquad\mathbf{v}_y:v_{y,1}<...<v_{y,n}<...<v_{y,N_{v_y}},\]
where we shall assume $N_v=N_{v_x}=N_{v_y}$ and $h_v=v_{x,m+1}-v_{x,m}=v_{y,n+1}-v_{y,n}$ for all $m$ and $n$. Over this velocity grid, we store the solution's coefficients $C^i_p(v_{x,m},v_{y,n},t)$ in a matrix $\mathbf{C}^i_p(t)\in\mathbb{R}^{N_v\times N_v}$ for each element $I_i$ ($1\leq i\leq N_x$) and each quadrature node $x^i_p$ ($1\leq p\leq k+1$). Under the associated Lagrangian basis, and discretized in velocity space, equation \eqref{eq: Cipmn} reduces to
\begin{equation}\label{eq: Cipmn_v2}
    \frac{d\mathbf{C}^{i}_{p}}{dt} = \mathcal{F}^p_+\left(\mathbf{C}^{i-1}_{\cdot},\mathbf{C}^{i}_{\cdot}\right) + \mathcal{F}^p_-\left(\mathbf{C}^{i}_{\cdot},\mathbf{C}^{i+1}_{\cdot}\right) + \frac{1}{\varepsilon^i_p}\left(\widetilde{\mathbf{M}}_{\mathbf{U}}(x^i_p,t) - \mathbf{C}^{i}_{p}\right),
\end{equation}
where $\widetilde{\mathbf{M}}_{\mathbf{U}}(x^i_p,t)\in\mathbb{R}^{N_v\times N_v}$ discretizes the local Maxwellian over the velocity grid and is constructed to ensure that the discrete mass, momentum, and energy moments of the collision operator vanishes (to be discussed shortly) and
\begin{align}\label{eq: Cip_v2_fluxes}
\begin{split}
    \mathcal{F}^p_+&\left(\mathbf{C}^{i-1}_{\cdot},\mathbf{C}^{i}_{\cdot}\right)\\
    &= -\frac{1}{h_xw_p}\max{(\mathbf{v}_{x},0)}*\sum\limits_{q=1}^{k+1}{\left( \mathbf{C}^{i}_{q}L_q\Big(\frac{1}{2}\Big)L_p\Big(\frac{1}{2}\Big) - \mathbf{C}^{i-1}_{q}L_q\Big(\frac{1}{2}\Big)L_p\Big(-\frac{1}{2}\Big) - \mathbf{C}^{i}_{q}w_q\frac{dL_p}{d\xi}(\xi_q)\right)},\\
    \mathcal{F}^p_-&\left(\mathbf{C}^{i}_{\cdot},\mathbf{C}^{i+1}_{\cdot}\right)\\
    &= -\frac{1}{h_xw_p} \min{(\mathbf{v}_{x},0)}*\sum\limits_{q=1}^{k+1}{\left( \mathbf{C}^{i+1}_{q}L_q\Big(-\frac{1}{2}\Big)L_p\Big(\frac{1}{2}\Big) - \mathbf{C}^{i}_{q}L_q\Big(-\frac{1}{2}\Big)L_p\Big(-\frac{1}{2}\Big) - \mathbf{C}^{i}_{q}w_q\frac{dL_p}{d\xi}(\xi_q)\right)},
\end{split}
\end{align}
with $*$ denoting the Hadamard (elementwise) product which we only use in the $\mathbf{v}_x$ dimension.
\subsection{Rank-adaptive discretization in velocity space}
Storing the $\mathbf{C}^i_p\in\mathbb{R}^{N_v\times N_v}$ for each element $I_i$ ($1\leq i\leq N_x$) and each quadrature node $x^i_s$ ($1\leq s\leq k+1$) increases both the storage and computational complexity. When solving kinetic equations in high-dimensional phase space, often to large times, this process can quickly become expensive. Low-rank matrix and tensor decompositions offer a remedy to this curse of dimensionality, which when coupled with arbitrarily high-order nodal DG methods, where the order scales with the number of quadrature nodes, leads to significant savings. A scalar function $f(v_x,v_y,t)$ is said to be rank-$r$ if it can be written as a sum of $r$ separable functions,
\begin{equation}\label{eq: schmidt}
    f(v_x,v_y,t) = \sum\limits_{l=1}^{r}{\alpha_{l}(t)\theta_{l}(v_x,t)\psi_{l}(v_y,t)}.
\end{equation}

If $\{\theta_1,...,\theta_r\}$ and $\{\psi_1,...,\psi_r\}$ are orthonormal sets and the scalars $\alpha_{\ell}$ are all real, non-negative, and unique up to re-ordering, then \eqref{eq: schmidt} is known as the \emph{Schmidt decomposition}. The singular value decomposition (SVD) is the discrete analogue to the Schmidt decomposition in two dimensions. Many kinetic equations, including the BGK equation, are known to exhibit low-rank solutions in velocity space, especially near equilibrium. This can be inferred from Property \ref{prop: equilibrium} since a Maxwellian distribution is a rank-1 function. Since $C^{i}_{p}(v_x,v_y,t)$ are the nodal values of the distribution function $f^i(x^i_p,v_x,v_y,t)$, we can assume that $\mathbf{C}^i_p$ can be approximated by a low-rank matrix; by a \textit{low-rank approximation}, we mean truncating the SVD according to a given tolerance on the singular values. Decomposing $\mathbf{C}^i_p$, we assume that
\begin{equation}\label{eq: Cip}
    \mathbf{C}^i_p(t) = \mathbf{\Theta}^i_p(t)\mathbf{S}^i_p(t)\mathbf{\Psi}^i_p(t)^T,
\end{equation}
where $\mathbf{\Theta}^i_p(t),\mathbf{\Psi}^i_p(t)\in\mathbb{R}^{N_v\times r(t)}$ have orthonormal columns, and $\mathbf{S}^i_p(t)\in\mathbb{R}^{r(t)\times r(t)}$ is diagonal with entries organized in decreasing magnitude. We note that solution \eqref{eq: Cip} is not necessarily the SVD, rather, the SVD is a special case that we usually consider. Solving the matrix differential equation associated with equation \eqref{eq: Cipmn_v2} reduces to summing several low-rank matrices, again recalling that the Maxwellian $M_{\mathbf{U}}$ is a rank-1 function in velocity space. Since naively adding together several low-rank matrices (or tensors) can increase the rank undesirably, we must do so in an efficient manner while maintaining as low a rank as possible. We use the QR-SVD approach used in \cite{guo2022low,nakao2025reduced}, described in Algorithm \ref{algo: qrsvd}. A low-rank approximation is obtained within tolerance $\vartheta>0$. It's important that the tolerance is smaller than the local truncation error of the time-stepping method so that the convergence order is unaffected; in practice, we choose $\vartheta\in[10^{-8},10^{-4}]$.

\begin{algorithm}[h!]
	\caption{Summing two low-rank matrices (using MATLAB syntax)}
	\label{algo: qrsvd}
		{\bf Input:} Low-rank matrices $\mathbf{A}_1=\mathbf{U}_1\mathbf{S}_1\mathbf{V}_1^T$ and $\mathbf{A}_2=\mathbf{U}_2\mathbf{S}_2\mathbf{V}_2^T$.\\
	    {\bf Output:} Low-rank approximation of $\mathbf{A}=\mathbf{A}_1+\mathbf{A}_2 \approx \mathbf{U}\mathbf{S}\mathbf{V}^T$.  
	\begin{algorithmic}[1]
	\State Augment the bases together.
        \Statex \texttt{[Qx,Rx] = qr([U1,U2],0);}
        \Statex \texttt{[Qy,Ry] = qr([V1,V2],0);}
        \State Compress by truncating the singular values to tolerance $\vartheta>0$ to remove redundant basis vectors.
        \Statex \texttt{S = blkdiag(S1,S2); \%direct sum of diagonal matrices S1 and S2}
        \Statex \texttt{[U,S,V] = svd(Rx*S*Ry',0);}
        \Statex \texttt{r = find(diag(S)>tol,1,`last');}
        \Statex \texttt{U = Qx*U(:,1:r);}
        \Statex \texttt{S = S(1:r,1:r);}
        \Statex \texttt{V = Qy*V(:,1:r);}
	\end{algorithmic}
\end{algorithm}

When a low-rank representation of the sum of three or more low-rank matrices is desired, Algorithm \ref{algo: qrsvd} can be applied in a hierarchical fashion. For example, if four low-rank matrices are to be summed: the first and second can be summed, the third and fourth can be summed, and then the two resulting low-rank matrices can be summed. This strategy is used to sum the low-rank matrices in equations \eqref{eq: Cipmn_v2} and \eqref{eq: Cip_v2_fluxes}.

Another advantage to a low-rank representation in velocity space is the reduced computational cost of integration. Consider a double integral over the velocity grid $\mathbf{v}_x\otimes\mathbf{v}_y$, whose Riemann sum would cost $\sim N_v^2$ flops. Alternatively, if the function has a low-rank decomposition, one simply has to sum (integrate) over each low-rank basis (dimension) separately before multiplying with the core matrix (coefficients), as seen in Figure \ref{fig: integratelowrank}. The computational complexity is then reduced to $\sim N_vr+r^2$ flops. Throughout the remainder of this paper, we refer to this process as the Low-Rank Double Integral (LRDI) function, \texttt{LRDI(U,S,V)}, which inputs a low-rank decomposition $\mathbf{USV}^T$, and outputs the resulting double integral.
\begin{figure}[h!]
    \centering
    \includegraphics[width=0.85\textwidth]{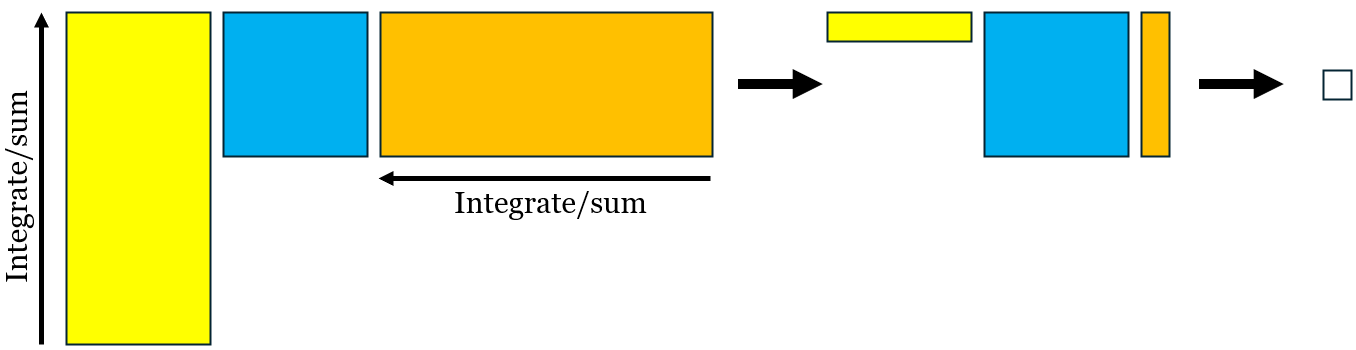}
    \caption{Integrating a low-rank matrix using the \texttt{LRDI} function.}
    \label{fig: integratelowrank}
\end{figure}

\begin{rem}
Although the solution can be decomposed in physical space, there is no guarantee that low-rank structures will exist for this model in $x$. Whereas, we can attempt to leverage the low-rank Maxwellian structures observed in velocity space, particularly near equilibrium.
\end{rem}

\subsection{Discrete moment equations}\label{sec: discretemomenteqns}
When solving equation \eqref{eq: Cipmn_v2}, the term from the BGK collision operator should be evolved implicitly due to its inherent stiffness. As such, we require the values of $\widetilde{\mathbf{M}}_{\mathbf{U}}(x^i_p,t)$ at the stages of the Runge-Kutta method. In other words, we need the moments of the distribution function at each quadrature node at each stage. Recalling that the coefficients are identical to nodal values of the solution, and referring to equation \eqref{eq: U}, we define the discrete moments at the nodal values,
\begin{equation}\label{eq: U_iq}
    \mathbf{U}^i_p(t)\coloneqq
    \begin{bmatrix*}
        n^i_p\\
        (nu_x)^i_p\\
        (nu_y)^i_p\\
        E^i_p
    \end{bmatrix*}
    =
    \begin{bmatrix*}
        \texttt{LRDI}(\mathbf{\Theta}^i_p,\mathbf{S}^i_p,\mathbf{\Psi}^i_p)\\
        \texttt{LRDI}(\mathbf{v}_x*\mathbf{\Theta}^i_p,\mathbf{S}^i_p,\mathbf{\Psi}^i_p)\\
        \texttt{LRDI}(\mathbf{\Theta}^i_p,\mathbf{S}^i_p,\mathbf{v}_y*\mathbf{\Psi}^i_p)\\
        \big(\texttt{LRDI}(\mathbf{v}_x^2*\mathbf{\Theta}^i_p,\mathbf{S}^i_p,\mathbf{\Psi}^i_p)+\texttt{LRDI}(\mathbf{\Theta}^i_p,\mathbf{S}^i_p,\mathbf{v}_y^2*\mathbf{\Psi}^i_p)\big)/2
    \end{bmatrix*},
\end{equation}
for $i=1,...,N_x$ and $p=1,...,k+1$, where $\mathbf{v}_x^2=\mathbf{v}_x*\mathbf{v}_x$ and $\mathbf{v}_y^2=\mathbf{v}_y*\mathbf{v}_y$. Letting $\mathbf{U}_h(\cdot,t)\in V_h^k$ approximate the discrete moments, $\mathbf{U}_h(x\in I_i,t) = \sum\limits_{q=1}^{k+1}{U_{p}^{i}(t)\phi_q^i(x)}$. Computing the moments of equation \eqref{eq: Cipmn_v2} yields the discrete moment equations of the 1d2v BGK model,
\begin{equation}\label{eq: momenteqns}
\frac{d\mathbf{U}^i_p}{dt} = \mathbf{P}^+_p(\mathbf{F}^{i-1,+}_{\cdot},\mathbf{F}^{i,+}_{\cdot}) + \mathbf{P}^-_p(\mathbf{F}^{i,-}_{\cdot},\mathbf{F}^{i+1,-}_{\cdot})
\end{equation}
where the discrete moments of the collision operator are ensured to vanish exactly (using quadrature corrected moments; see subsection \ref{sec: QCM}), and the upwind flux terms are

\begin{align}
\begin{split}
    \mathbf{P}^+_p(\mathbf{F}^{i-1,+}_{\cdot},\mathbf{F}^{i,+}_{\cdot}) =&-\frac{1}{h_xw_p}\sum\limits_{q=1}^{k+1}{\left(\mathbf{F}^{i,+}_qL_q\Big(\frac{1}{2}\Big)L_p\Big(\frac{1}{2}\Big) - \mathbf{F}^{i-1,+}_qL_q\Big(\frac{1}{2}\Big)L_p\Big(-\frac{1}{2}\Big) - \mathbf{F}^{i,+}_qw_q\frac{dL_p}{d\xi}(\xi_q)\right)},\\
    \mathbf{P}^-_p(\mathbf{F}^{i,-}_{\cdot},\mathbf{F}^{i+1,-}_{\cdot})=&- \frac{1}{h_xw_p}\sum\limits_{q=1}^{k+1}{\left(\mathbf{F}^{i+1,-}_qL_q\Big(-\frac{1}{2}\Big)L_p\Big(\frac{1}{2}\Big) - \mathbf{F}^{i,-}_qL_q\Big(-\frac{1}{2}\Big)L_p\Big(-\frac{1}{2}\Big) - \mathbf{F}^{i,-}_qw_q\frac{dL_p}{d\xi}(\xi_q)\right)},
\end{split}
\end{align}
with
\begin{align}\label{eq: moment_fluxes}
\begin{split}
    \mathbf{F}_q^{i,+} &=
    \begin{bmatrix*}
        \texttt{LRDI}(\mathbf{v}_x^+*\mathbf{\Theta}^i_p,\mathbf{S}^i_p,\mathbf{\Psi}^i_p)\\
        \texttt{LRDI}((\mathbf{v}_x^+)^2*\mathbf{\Theta}^i_p,\mathbf{S}^i_p,\mathbf{\Psi}^i_p)\\
        \texttt{LRDI}(\mathbf{v}_x^+*\mathbf{\Theta}^i_p,\mathbf{S}^i_p,\mathbf{v}_y*\mathbf{\Psi}^i_p)\\
        \big(\texttt{LRDI}((\mathbf{v}_x^+)^3*\mathbf{\Theta}^i_p,\mathbf{S}^i_p,\mathbf{\Psi}^i_p)+\texttt{LRDI}(\mathbf{v}_x^+*\mathbf{\Theta}^i_p,\mathbf{S}^i_p,\mathbf{v}_y^2*\mathbf{\Psi}^i_p)\big)/2
    \end{bmatrix*}
\end{split},
\\
\begin{split}
    \mathbf{F}_q^{i,-} &=
    \begin{bmatrix*}
        \texttt{LRDI}(\mathbf{v}_x^-*\mathbf{\Theta}^i_p,\mathbf{S}^i_p,\mathbf{\Psi}^i_p)\\
        \texttt{LRDI}((\mathbf{v}_x^-)^2*\mathbf{\Theta}^i_p,\mathbf{S}^i_p,\mathbf{\Psi}^i_p)\\
        \texttt{LRDI}(\mathbf{v}_x^-*\mathbf{\Theta}^i_p,\mathbf{S}^i_p,\mathbf{v}_y*\mathbf{\Psi}^i_p)\\
        \big(\texttt{LRDI}((\mathbf{v}_x^-)^3*\mathbf{\Theta}^i_p,\mathbf{S}^i_p,\mathbf{\Psi}^i_p)+\texttt{LRDI}(\mathbf{v}_x^-*\mathbf{\Theta}^i_p,\mathbf{S}^i_p,\mathbf{v}_y^2*\mathbf{\Psi}^i_p)\big)/2
    \end{bmatrix*}
\end{split},
\end{align}
where $\mathbf{v}_x^+=\text{max}(\mathbf{v}_x,\mathbf{0})$ and $\mathbf{v}_x^-=\text{min}(\mathbf{v}_x,\mathbf{0})$. As a result of the collisional invariants for the BGK operator (see Property \ref{prop: collisional_invariants}), the moment equations \eqref{eq: momenteqns} only have the nonstiff transport term and can be evolved explicitly. Thus, we can solve for the discrete moments at the nodal values with the same explicit Runge-Kutta method from the IMEX method that will update the coefficients of the distribution function.

\subsection{IMEX Runge-Kutta discretization in time}\label{sec: IMEXRK}
We solve the BGK equation using a second-order IMEX RK method, with higher-order methods a straightforward extension; see \cite{ascher1997implicit} for a comprehensive collection of IMEX RK methods. DIRK methods have been shown to be a good choice when evolving the BGK operator, which can be quite stiff for small $\varepsilon$; and naturally, the transport term is evolved explicitly.

While the coefficients in equation \eqref{eq: Cipmn_v2} are evolved using the IMEX RK method, the discrete moments in equation \eqref{eq: momenteqns} will be updated using the same explicit RK method that's utilized in the IMEX method. Broadly speaking, equation \eqref{eq: momenteqns} is solved to the first stage $t^{(1)}$ of the explicit RK method; those moments are fed into equation \eqref{eq: Cipmn_v2} in order to solve for the coefficients at the first stage $t^{(1)}$ of the IMEX RK method; the coefficients are used to compute the flux terms \eqref{eq: moment_fluxes} in order to solve for the discrete moments in equation \eqref{eq: momenteqns} at the second stage $t^{(2)}$; and so forth until the coefficients at the next time-step are computed. It is imperative that the times of the stages in both the explicit RK and IMEX RK methods coincide.

We want to update the solution from time $t^{\ell}$ to time $t^{\ell+1}$, with $\Delta t=t^{\ell+1}-t^{\ell}$ for $\ell=0,1,...,N_t$. We consider the following stiffly accurate second-order IMEX RK (and associated explicit RK) scheme:
\begin{subequations}
\begin{align}
    \begin{split}\label{eq: U_stage1}
    (\mathbf{U}^i_p)^{(1)} &= (\mathbf{U}^i_p)^{(0)} + \gamma\Delta t(\mathbf{P}^+_p+\mathbf{P}^-_p)^{(0)}
    \end{split}
    \\
    \begin{split}\label{eq: C_stage1}
    (\mathbf{C}^{i}_{p})^{(1)} &= (\mathbf{C}^{i}_{p})^{(0)} + \gamma\Delta t(\mathcal{F}^p_+ + \mathcal{F}^p_-)^{(0)} + \frac{\gamma\Delta t}{\varepsilon^i_p}\left(\widetilde{\mathbf{M}}_{\mathbf{U}}(x^i_p,t^{(1)}) - (\mathbf{C}^{i}_{p})^{(1)}\right)
    \end{split}
    \\
    \begin{split}\label{eq: U_stage2}
    (\mathbf{U}^i_p)^{(2)} &= (\mathbf{U}^i_p)^{(0)} + \delta\Delta t(\mathbf{P}^+_p+\mathbf{P}^-_p)^{(0)} + (1-\delta)\Delta t(\mathbf{P}^+_p+\mathbf{P}^-_p)^{(1)}
    \end{split}
    \\
    \begin{split}\label{eq: C_stage2}
    (\mathbf{C}^{i}_{p})^{(2)} &= (\mathbf{C}^{i}_{p})^{(0)} + \delta\Delta t(\mathcal{F}^p_+ + \mathcal{F}^p_-)^{(0)} + (1-\delta)\Delta t(\mathcal{F}^p_+ + \mathcal{F}^p_-)^{(1)}\\
    &\quad + \frac{(1-\gamma)\Delta t}{\varepsilon^i_p}\left(\widetilde{\mathbf{M}}_{\mathbf{U}}(x^i_p,t^{(1)}) - (\mathbf{C}^{i}_{p})^{(1)}\right) + \frac{\gamma\Delta t}{\varepsilon^i_p}\left(\widetilde{\mathbf{M}}_{\mathbf{U}}(x^i_p,t^{(2)}) - (\mathbf{C}^{i}_{p})^{(2)}\right)
    \end{split}
\end{align}
\label{eqn:numerical_solution}
\end{subequations}
where $\gamma=1-\sqrt{2}/2$, $\delta=1-1/(2\gamma)$, and superscripts are used to denote numerical solutions at the stages $t^{(0)}=t^{\ell}$, $t^{(1)}=t^{\ell}+\gamma\Delta t$ and $t^{(2)}=t^{\ell+1}$. When solving equation \eqref{eq: C_stage1}, move all $(\mathbf{C}^{i}_{p})^{(1)}$ terms to the left side of the equation and divide by $1+\gamma\Delta t/\varepsilon^i_p$; do the same to equation \eqref{eq: C_stage2}.

There are many advantages that come from combining nodal DG and low-rank methods into a single framework. The proposed scheme is highly parallelizable and can be arbitrarily high-order in physical space due to the nodal DG method. The computational cost of several matrix computations that are required for the kinetic and moment equations is reduced when working with low-rank decompositions in velocity space. And, the distribution function can be evolved using high-order IMEX discretizations. It's also worth noting that by feeding the distribution function into the integral form flux terms of the moment equations, we avoid the moment closure problem that fluid solvers often need to address.


\subsection{Post-processing slope limiters}
Several kinetic simulations evolve solutions that exhibit discontinuities or sharp gradients in physical space, hence necessitating numerical schemes that are robust to spurious oscillations. In the presence of discontinuities, oscillations can occur when evolving both the kinetic solution \textit{and} the moments in \eqref{eq: U_stage1}-\eqref{eq: C_stage2}; the oscillations will remain after computing the corrected moments \eqref{eq: nonlinear_qcm_system}.

Although there is a rich literature of high resolution methods \cite{shu2009high,toro2013riemann} (e.g., flux and slope limiters) in the full-rank setting, rigorously extending these methods to low-rank frameworks remains a relatively young and underdeveloped area of research. In particular, efficiently applying a limiter to either the kinetic fluxes or kinetic solution is non-trivial mainly due to the enforcement of a low-rank decomposition. To work around this, we follow the strategy from \cite{xiong2015high}, in which a slope limiter is applied to the system states/moments as a post-processing step. Referring to equations \eqref{eq: U_stage1}-\eqref{eq: C_stage2}, we performed the following steps at each stage: (1) update the moments, (2) apply a slope limiter to the updated moments, and (3) use the post-processed moments in the kinetic equation. We apply a slope limiter to the moments instead of the kinetic solution for two reasons. First, discontinuities in the kinetic solution ultimately come from discontinuities in the system states. Since the BGK operator is defined in terms of the macroscopic quantities (in $M_{\mathbf{U}}$), reducing oscillations in the moments is critical. Second, the moment system is smaller than the kinetic solution, and the moments are not held in a low-rank decomposition.

To apply a slope limiter, we must first identify \textit{troubled elements}, that is, elements in which there are large oscillations. Following which, a limiter of choice is applied to the system states/moments. Similar to the strategy in \cite{xiong2015high,zhong2013simple}, we identify the troubled elements using the standard minmod limiter \cite{harten1983high}. We provide a brief recap of the minmod limiter, also provided in \cite{zhong2013simple}, for completeness. Let $\bar{U}_j$ denote the cell average of a scalar function $U$ over element $I_j$ ($1\leq j\leq N_x$); $U$ could be any of the macroscopic quantities, and the cell averages are computed using the Gauss-Legendre quadrature since the nodal values are known. Further denote $\tilde{U}_j=U_{j+\frac{1}{2}}^--\bar{U}_j$ and $\tilde{\bar{U}}_j=\bar{U}_j-U_{j-\frac{1}{2}}^+$; the limits at the cell boundary are computed using the nodal basis (see the footnote in subsection \ref{sec: discretemomenteqns}). The \textit{minmod function} is defined by
\begin{equation}\label{eq: minmod}
    m(a_1,...,a_{\nu})=
    \begin{cases}
        s\min\limits_{1\leq\ell\leq\nu}{|a_{\ell}|},&\text{if}\ s=\text{sign}(a_1)=...=\text{sign}(a_{\nu}),\\
        0,&\text{otherwise}.
    \end{cases}
\end{equation}

Using the minmod function to modify the differences $\tilde{U}_j$ and $\tilde{\bar{U}}_j$,
\begin{equation}\label{eq: minmod_activate}
    \tilde{U}^{(\text{mod})}_j=m(\tilde{U}_j,\bar{U}_{j+1}-\bar{U}_j,\bar{U}_j-\bar{U}_{j-1}),\qquad \tilde{\bar{U}}_j^{(\text{mod})}=m(\tilde{\bar{U}}_j,\bar{U}_{j+1}-\bar{U}_j,\bar{U}_j-\bar{U}_{j-1}).
\end{equation}

If the minmod function \eqref{eq: minmod} is activated in either expression in equation \eqref{eq: minmod_activate}, then the element is considered troubled, and we apply a slope limiter of choice to post-process the nodal values of the moments in each troubled element. In our numerical experiments, we compare the performance of two slope limiters: the \textit{minmod limiter} used in \cite{xiong2015high} (also see \cite{cockburn1999discontinuous}, Section 2.4.3, limiter C), and the \textit{WENO limiter} proposed in \cite{zhong2013simple}. Naturally, other limiters can also be used; regardless of the slope limiter that's applied to nodal values, we use the minmod limiter just described to identify the troubled cells. Applying either limiter will approximate the macroscopic quantity at the cell boundaries, $V_{j+\frac{1}{2}}^-\approx U_{j+\frac{1}{2}}^-$ and $V_{j-\frac{1}{2}}^+\approx U_{j-\frac{1}{2}}^+$.

\begin{rem}\label{ref: limiters}
The minmod limiter presented in \cite{xiong2015high} only approximates the moments at the element boundaries, using the linear interpolant between these two values to approximate the moments at the nodal values. While this introduces a second-order error, it is only applied over the troubled cells. Whereas, the WENO limiter proposed in \cite{zhong2013simple} defines the reconstruction polynomial over the entire element and can thus be used to get high-order approximations at the nodal values.
\end{rem}


\subsection{Quadrature corrected moments}\label{sec: QCM}
We enforce the vanishing of the discrete moments of the collision operator in equation \eqref{eq: momenteqns} by adopting a specific discrete definition of the Maxwellian:
\begin{align}
    \label{eqn:special_mu}
    \widetilde{{\bf M}}_{\bf U} \left( x^i_p; {\cal M}^{i}_{p} \right) = 
    \left({\cal N}_{v_x}\right)^{i}_p {\cal I}^{i}_{p} \left( {\cal N}^{\text{T}}_{v_y} \right)^{i}_{p}.
\end{align}

Here, ${\cal M}^{i}_{p} = \left\{n^i_{M,p}, {\bf u}^i_{M,p}, T^i_{M,p} \right\}$ are the \textit{quadrature corrected moments} (QCM) \cite{taitano_jcp_2017_ep_rfp}. For the rest of this subsection, we drop the $x$ cell and Gauss quadrature indices for brevity. The functions ${\cal N}_{v_x}({\bf v_x}; {\cal M}) = e^{-\frac{\left(\mathbf{v}_x - u_{x,M} \right)^2}{2T}} \in \mathbb{R}^{N_v}_+$ and ${\cal N}_{v_y}({\bf v_y}; {\cal M}) = e^{-\frac{\left(\mathbf{v}_y - u_{y,M} \right)^2}{2T}} \in \mathbb{R}^{N_v}_+$ represent rank-1 basis functions of the Maxwellian in $v_x$ and $v_y$, respectively, while ${\cal I} = \frac{n_M}{2 \pi  T_M} \in \mathbb{R}_+$ is a normalization factor. The role of ${\cal M}$ is to guarantee that the discrete moments of the collision operator, evaluated with the Maxwellian parameterized by ${\cal M}$, vanish exactly. 

The corrected quantities are obtained by solving a nonlinear system for ${\cal M}$ at each Gauss–Legendre quadrature node, using a Newton method with the code found in \cite{Mjaavatten2025}. The $k^{\text{th}}$ Newton iteration is given by ${\cal M}^{k+1} = {\cal M}^{k} + \delta{\cal M}^{k}$, with the update $\delta {\cal M}^{k} = -\left(\mathbb{J}^{k}\right)^{-1} {\bf R}^{k}$. Here, $\mathbb{J}^{k} \equiv \frac{\partial {\bf R}^{k}}{\partial {\cal M}^k} \in \mathbb{R}^{4\times 4}$ is the Jacobian matrix, and ${\bf R}^{k} \equiv {\bf R}\left( {\cal M}^{k} \right) = \{R^k_{n}, R^k_{nu_x}, R^k_{nu_y}, R^k_{E}\}^T \in \mathbb{R}^4$ is the residual vector with components defined as
\begin{subequations}
\label{eq: nonlinear_qcm_system}
    \begin{align}
        R_{n_M} = &n^i_p - \texttt{LRDI}\!\left( \left({\cal N}_{v_x}\right)^{i}_p , {\cal I}^{i}_{p} , \left( {\cal N}_{v_y} \right)^{i}_{p} \right), \\
        R_{nu_x} = &(nu_x)^i_p - \texttt{LRDI}\!\left( {\bf v_x} *  \left( {\cal N}_{v_x}\right)^{i}_p , {\cal I}^{i}_{p} , \left( {\cal N}_{v_y} \right)^{i}_{p} \right), \\
        R_{nu_y} = &(nu_y)^i_p - \texttt{LRDI}\!\left( \left({\cal N}_{v_x}\right)^{i}_p , {\cal I}^{i}_{p} , {\bf v_y} * \left(  {\cal N}_{v_y} \right)^{i}_{p} \right), \\
        R_{E} = &E^i_p - \tfrac{1}{2}\Big[
            \texttt{LRDI}\!\left(\mathbf{v}_x^2 * \left({\cal N}_{v_x}\right)^{i}_p , {\cal I}^{i}_{p} , \left( {\cal N}_{v_y} \right)^{i}_{p} \right)
            + \texttt{LRDI}\!\left(\left({\cal N}_{v_x}\right)^{i}_p , {\cal I}^{i}_{p} , {\bf v^2_y} * \left( {\cal N}_{v_y} \right)^{i}_{p}\right)\Big].
    \end{align}
\end{subequations}
The Newton solver iterates until the stopping criterion, $R^k \le 10^{-13} R^0$, is satisfied, where $R^k \equiv \left| {\bf R}({\cal M}^k) \right|$. We note that the Newton solver typically converges to the specified tolerance in a few iterations and constitutes negligible cost of the overall algorithm. Due to the QCM procedure, the moments computed from equation \eqref{eq: U_stage2} are conserved up to machine precision. However, the discrete moments of the solution computed from equation \eqref{eq: C_stage2} preserves this conservation only up to the truncation tolerance. This is demonstrated in our numerical experiments.

\begin{proposition}
    For the BGK equation \eqref{eq: BGK}, we consider the numerical scheme \eqref{eqn:numerical_solution} equipped with the quadrature corrected moments. In the asymptotic limit $\varepsilon\to 0$, the numerical solution for the distribution function converges to the corrected local Maxwellian for all $1\leq i\leq N_x$ and $1\leq p\leq k+1$.
\end{proposition}
\begin{proof}
Let $i$ and $p$ be fixed, with $1\leq i\leq N_x$ and $1\leq p\leq k+1$ respectively. We define the $(\mathbf{A}^{i}_p)^{(\star)}$ terms, with $\star\in\{0,1\}$,
\begin{subequations}
\begin{align}
\begin{split}
    (\mathbf{A}^i_p)^{(0)}&=(\mathbf{C}^{i}_{p})^{(0)} + \gamma\Delta t(\mathcal{F}^p_+ + \mathcal{F}^p_-)^{(0)},
\end{split}
\\
\begin{split}
    (\mathbf{A}^i_p)^{(1)}&=(\mathbf{C}^{i}_{p})^{(0)} + \delta\Delta t(\mathcal{F}^p_+ + \mathcal{F}^p_-)^{(0)} + (1-\delta)\Delta t(\mathcal{F}^p_+ + \mathcal{F}^p_-)^{(1)}. 
\end{split}
\end{align}
\end{subequations}
Solving equations \eqref{eq: C_stage1} and \eqref{eq: C_stage2},
\begin{subequations}
\begin{align}
    \begin{split}\label{eq:C1}
    (\mathbf{C}^{i}_{p})^{(1)} &=\frac{\varepsilon^i_p}{\varepsilon^i_p+\gamma\Delta t}(\mathbf{A}^i_p)^{(0)} + \frac{\gamma\Delta t}{\varepsilon^i_p+\gamma\Delta t}\widetilde{\mathbf{M}}_{\mathbf{U}}(x^i_p,t^{(1)}),
    \end{split}
    \\
    \begin{split}\label{eq:C2}
    (\mathbf{C}^{i}_{p})^{(2)} &=\frac{\varepsilon^i_p}{\varepsilon^i_p+\gamma\Delta t}(\mathbf{A}^i_p)^{(1)} + \frac{(1-\gamma)\Delta t}{\varepsilon^i_p+\gamma\Delta  t}\left(\widetilde{\mathbf{M}}_{\mathbf{U}}(x^i_p,t^{(1)}) - (\mathbf{C}^{i}_{p})^{(1)}\right) + \frac{\gamma\Delta t}{\varepsilon^i_p+\gamma\Delta t}\widetilde{\mathbf{M}}_{\mathbf{U}}(x^i_p,t^{(2)}).
    \end{split}
\end{align}
\end{subequations}

Substituting \eqref{eq:C1} into \eqref{eq:C2},
\begin{equation}
    (\mathbf{C}^{i}_{p})^{(2)} =\frac{\varepsilon^i_p}{\varepsilon^i_p+\gamma\Delta t}(\mathbf{A}^i_p)^{(1)} + \frac{\varepsilon^i_p(1-\gamma)\Delta t}{(\varepsilon^i_p+\gamma\Delta  t)^2}\left(\widetilde{\mathbf{M}}_{\mathbf{U}}(x^i_p,t^{(1)}) - (\mathbf{A}^i_p)^{(0)}\right) + \frac{\gamma\Delta t}{\varepsilon^i_p+\gamma\Delta t}\widetilde{\mathbf{M}}_{\mathbf{U}}(x^i_p,t^{(2)}). 
\end{equation}
Taking the limit $\varepsilon\to 0$ (and hence $\varepsilon^i_p\to 0$), the proof is complete since the coefficients are identical to the nodal values.
\end{proof}

Due to the velocity discretization and QCM procedure to correct the moments discretely, showing the AP property is more subtle and a topic of future investigation. In summary, we evolve the solution according to the following flow chart; subscript $\star$ denotes the moments \textit{before} the post-processing slope limiter.
\[\left\{(\mathbf{C}^i_p)^{(0)}\right\}\xrightarrow[\text{Eqn. \eqref{eq: U_stage1}}]{}\left\{(\mathbf{U}^i_p)^{(1)}_{\star}\right\}\xrightarrow[\text{Slope limiter}]{}\left\{(\mathbf{U}^i_p)^{(1)}\right\}\xrightarrow[\text{QCM}]{\text{Correct with}}\left\{(\mathbf{U}^i_p)^{(1)}\right\}\xrightarrow[\text{Eqn. \eqref{eq: C_stage1}}]{}\left\{(\mathbf{C}^i_p)^{(1)}\right\}\]
\[\left\{(\mathbf{C}^i_p)^{(1)}\right\}\xrightarrow[\text{Eqn. \eqref{eq: U_stage2}}]{}\left\{(\mathbf{U}^i_p)^{(2)}_{\star}\right\}\xrightarrow[\text{Slope limiter}]{}\left\{(\mathbf{U}^i_p)^{(2)}\right\}\xrightarrow[\text{QCM}]{\text{Correct with}}\left\{(\mathbf{U}^i_p)^{(2)}\right\}\xrightarrow[\text{Eqn. \eqref{eq: C_stage2}}]{}\left\{(\mathbf{C}^i_p)^{(2)}\right\}\]

\section{Numerical experiments}\label{sec: numerics}

We test the proposed nodal DG method for solving the BGK equation \eqref{eq: BGK}. The proposed scheme is tested on different sets of initial conditions to demonstrate its high-order accuracy, computational efficiency, and robustness. As a result of the QCM enforcing the discrete conservation of the moments, the scheme conserves mass, momentum, and energy up to the truncation tolerance $\vartheta$. Unless otherwise stated, the truncation tolerance is set to $\vartheta=10^{-6}$; larger tolerances can be used but should be smaller than the local truncation error. We consider piecewise polynomial spaces with degree at most $k=0,1,2$, corresponding to $k+1$ point Gauss-Legendre quadratures. Let NDG1, NDG2, and NDG3 signify the nodal DG discretizations with order $k+1$, for $k=0,1,2$. We only consider the second-order IMEX RK scheme in Section \ref{sec: IMEXRK}, and the time-stepping size is defined according to a CFL condition $\Delta t=C_{\text{CFL}}h_x/V_{\text{max}}$, with $C_{\text{CFL}}=1/(2k+3)$. The velocity domain $[-V_{\text{max}},V_{\text{max}}]$ is set large enough to allow sufficient decay of the distribution function. And, when determining the initial states of the discontinuous tests, the specific heat ratio $\gamma=(d+2)/d=2$, where $d=2$ is the translational degrees of freedom in velocity space. All numerical tests were run in MATLAB R2024a on a MacBook Pro with the Apple M2 Pro Processor (3.49 GHz).

\subsection{Smooth example}

We demonstrate the accuracy and computational complexity with a smooth solution in which the initial condition is a Maxwellian distribution determined by the following initial states:
\begin{equation}\label{eq: smoothIC}
    n(x,0)=1+0.5\sin{(x)},\qquad u_x(x,0)=u_y(x,0)=0,\qquad T(x,0)=1,
\end{equation}
with periodic boundary conditions in $x$, and physical domain $\Omega_x=[0,2\pi]$. The velocity domain is set with $V_{\text{max}}=10$ and discretized using $N_v=100$ uniformly spaced points in each direction. We use a non-uniform mesh in physical space $x$ with an even number of elements $N_x$. The mesh is constructed by expressing $\Omega_x$ as two disjoint subintervals $[0,4\pi/3]\cup[4\pi/3,2\pi]$, using $N_x/2$ uniform cells for both subregions.

Since there is no closed form exact solution, we test the order of accuracy under mesh refinement as in \cite{xiong2015high}, where the $L^1$ error is defined as the difference of the numerical solution between consecutive refinements,
\begin{equation}
    \text{$L^1$ error of $f[h_x]$} = \frac{h_v^2}{2\pi}\sum\limits_{i,j,k}{\int_{I_i}{\Big|f[h_x](x,v_{x,j},v_{y,k},t_f) - f[h_x/2](x,v_{x,j},v_{y,k},t_f)\Big|dx}},
\end{equation}
where $t_f$ is the final time, and the integrals in $x$ are computed using Guassian quadrature with the nodal values. To demonstrate the spatial order of accuracy from the nodal DG discretization in $x$, Table \ref{table: errortable1} presents the $L^1$ error, with $t_f=0.001$, $\varepsilon\in\{1,10^{-2},10^{-6}\}$, and no limiter applied. Due to the small final time, the error is dominated by the spatial discretization, and the expected order of accuracy is observed. We only show results for NDG2 and NDG3, but we observed first-order accuracy with NDG1 in our experiments.

\begin{table}[h!]
\centering
\caption{Convergence study using IC \eqref{eq: smoothIC}, velocity mesh $N_v=100$, and final time $t_f=0.001$. Non-uniform mesh and no limiter applied.}
\label{table: errortable1}
\begin{tabular}{cccccccccccccccc}
\cline{1-15}
     &  &       &  & \multicolumn{3}{c}{$\varepsilon=1.0$} &  & \multicolumn{3}{c}{$\varepsilon=10^{-2}$} &  & \multicolumn{3}{c}{$\varepsilon=10^{-6}$} &  \\ \cline{5-7} \cline{9-11} \cline{13-15}
     &  & $N_x$ &  & $L^1$ error      &       & order      &  & $L^1$ error        &        & order       &  & $L^1$ error        &        & order       &  \\ \cline{1-15}
     &  & 16    &  & 3.60E-02         &       & -          &  & 3.60E-02           &        & -           &  & 3.60E-02           &        & -           &  \\
     &  & 32    &  & 1.79E-02         &       & 1.01       &  & 1.79E-02           &        & 1.01        &  & 1.79E-02           &        & 1.01        &  \\
NDG2 &  & 64    &  & 1.42E-04         &       & 2.00       &  & 1.42E-04           &        & 2.00        &  & 1.42E-04           &        & 2.00        &  \\
     &  & 128   &  & 3.54E-05         &       & 2.00       &  & 3.54E-05           &        & 2.00        &  & 3.54E-05           &        & 2.00        &  \\
     &  & 256   &  & 8.85E-06         &       & 2.00       &  & 8.85E-06           &        & 2.00        &  & 8.85E-06           &        & 2.00        &  \\ \cline{1-15}
     &  & 16    &  & 1.03E-04         &       & -          &  & 1.03E-04           &        & -           &  & 1.03E-04           &        & -           &  \\
     &  & 32    &  & 1.28E-05         &       & 3.00       &  & 1.28E-05           &        & 3.00        &  & 1.28E-05           &        & 3.00        &  \\
NDG3 &  & 64    &  & 1.61E-06         &       & 2.99       &  & 1.61E-06           &        & 2.99        &  & 1.61E-06           &        & 2.99        &  \\
     &  & 128   &  & 2.03E-07         &       & 2.99       &  & 2.03E-07           &        & 2.99        &  & 2.04E-07           &        & 2.99        &  \\
     &  & 256   &  & 2.58E-08         &       & 2.98       &  & 2.58E-08           &        & 2.98        &  & 2.59E-08           &        & 2.98        &  \\ \cline{1-15}
\end{tabular}
\end{table}

Although the slope limiter is typically not activated for this smooth problem since there are no troubled elements, we test the accuracy when applying the slope limiter to each element. Table \ref{table: errortable_weno} presents the $L^1$-error using the same parameters as the previous example, while also applying the WENO limiter with parameters $\gamma_0=\gamma_2=0.5\times10^{-8}$ and $\gamma_1=1-10^{-8}$. As stated in \cite{zhong2013simple}, when choosing the linear weights for the WENO method, larger values of the ratios $\gamma_1/\gamma_0$ and $\gamma_2/\gamma_0$ are known to yield better results for smooth problems. The results in Table \ref{table: errortable_minmod} using the minmod limiter only show second-order accuracy since the minmod limiter uses a linear interpolant to approximate the moments at the nodal values, as per Remark \ref{ref: limiters}. The collision operator is stronger for smaller $\varepsilon$. Thus, the error incurred by approximating the moments at the nodal values using the linear interpolant plays a bigger role, explaining why we see the accuracy using NDG3 drop to second-order faster for $\varepsilon=10^{-6}$.

\begin{table}[h!]
\centering
\caption{Convergence study using IC \eqref{eq: smoothIC}, velocity mesh $N_v=100$, and final time $t_f=0.001$. Non-uniform mesh and WENO limiter applied.}
\label{table: errortable_weno}
\begin{tabular}{cccccccccccccccc}
\cline{1-15}
     &  &       &  & \multicolumn{3}{c}{$\varepsilon=1.0$} &  & \multicolumn{3}{c}{$\varepsilon=10^{-2}$} &  & \multicolumn{3}{c}{$\varepsilon=10^{-6}$} &  \\ \cline{5-7} \cline{9-11} \cline{13-15}
     &  & $N_x$ &  & $L^1$ error      &       & order      &  & $L^1$ error        &        & order       &  & $L^1$ error        &        & order       &  \\ \cline{1-15}
     &  & 16    &  & 1.78E-03         &       & -          &  & 1.99E-03           &        & -           &  & 4.37E-03           &        & -           &  \\
     &  & 32    &  & 4.44E-04         &       & 2.00       &  & 4.68E-04           &        & 2.09        &  & 7.87E-04           &        & 2.47        &  \\
NDG2 &  & 64    &  & 1.11E-04         &       & 2.00       &  & 1.12E-04           &        & 2.06        &  & 1.38E-04           &        & 2.51        &  \\
     &  & 128   &  & 2.77E-05         &       & 2.00       &  & 2.77E-05           &        & 2.02        &  & 2.88E-05           &        & 2.26        &  \\
     &  & 256   &  & 6.92E-06         &       & 2.00       &  & 6.92E-06           &        & 2.00        &  & 6.96E-06           &        & 2.05        &  \\ \cline{1-15}
     &  & 16    &  & 5.79E-05         &       & -          &  & 3.02E-04           &        & -           &  & 2.69E-03           &        & -           &  \\
     &  & 32    &  & 7.29E-06         &       & 2.99       &  & 3.54E-05           &        & 3.09        &  & 3.71E-04           &        & 2.86        &  \\
NDG3 &  & 64    &  & 9.13E-07         &       & 3.00       &  & 2.76E-06           &        & 3.68        &  & 3.34E-05           &        & 3.47        &  \\
     &  & 128   &  & 1.15E-07         &       & 2.99       &  & 1.86E-07           &        & 3.90        &  & 1.61E-06           &        & 4.38        &  \\
     &  & 256   &  & 1.46E-08         &       & 2.98       &  & 1.68E-08           &        & 3.47        &  & 6.32E-08           &        & 4.67        &  \\ \cline{1-15}
\end{tabular}
\end{table}

\begin{table}[h!]
\centering
\caption{Convergence study using IC \eqref{eq: smoothIC}, velocity mesh $N_v=100$, and final time $t_f=0.001$. Non-uniform mesh and minmod limiter applied.}
\label{table: errortable_minmod}
\begin{tabular}{cccccccccccccccc}
\cline{1-15}
     &  &       &  & \multicolumn{3}{c}{$\varepsilon=1.0$} &  & \multicolumn{3}{c}{$\varepsilon=10^{-2}$} &  & \multicolumn{3}{c}{$\varepsilon=10^{-6}$} &  \\ \cline{5-7} \cline{9-11} \cline{13-15}
     &  & $N_x$ &  & $L^1$ error      &       & order      &  & $L^1$ error        &        & order       &  & $L^1$ error        &        & order       &  \\ \cline{1-15}
     &  & 16    &  & 2.28E-03         &       & -          &  & 2.54E-03           &        & -           &  & 5.89E-03           &        & -           &  \\
     &  & 32    &  & 5.67E-04         &       & 2.01       &  & 6.04E-04           &        & 2.02        &  & 1.20E-03           &        & 2.30        &  \\
NDG2 &  & 64    &  & 1.42E-04         &       & 2.00       &  & 1.50E-04           &        & 2.01        &  & 2.65E-04           &        & 2.17        &  \\
     &  & 128   &  & 3.54E-05         &       & 2.00       &  & 3.75E-05           &        & 2.00        &  & 6.25E-05           &        & 2.08        &  \\
     &  & 256   &  & 8.85E-06         &       & 2.00       &  & 9.36E-06           &        & 2.00        &  & 1.54E-05           &        & 2.02        &  \\ \cline{1-15}
     &  & 16    &  & 1.04E-04         &       & -          &  & 5.30E-04           &        & -           &  & 4.70E-03           &        & -           &  \\
     &  & 32    &  & 1.33E-05         &       & 2.96       &  & 8.36E-05           &        & 2.66        &  & 7.75E-04           &        & 2.60        &  \\
NDG3 &  & 64    &  & 1.74E-06         &       & 2.94       &  & 1.46E-05           &        & 2.52        &  & 1.44E-04           &        & 2.43        &  \\
     &  & 128   &  & 2.30E-07         &       & 2.92       &  & 2.87E-06           &        & 2.35        &  & 3.03E-05           &        & 2.24        &  \\
     &  & 256   &  & 3.19E-08         &       & 2.85       &  & 6.27E-07           &        & 2.19        &  & 7.28E-06           &        & 2.06        &  \\ \cline{1-15}
\end{tabular}
\end{table}

\begin{table}[h!]
\centering
\caption{Comparison of CPU runtimes using initial condition \eqref{eq: smoothIC}, spatial mesh $N_x=40$, and final time $t_f=0.1$. Runtimes are normalized to the runtime of the $N_v=32$, NDG1 case, which was 7.8985 seconds.}
\label{table: runtimes}
\begin{tabular}{|c|c|c|c|}
\hline
$N_v$ & NDG1  & NDG2 & NDG3 \\ \hline
32     & 1.000 & 3.910 & 9.624 \\
64     & 0.996  & 4.005  & 9.779 \\
128    & 1.040 & 4.259 & 10.815 \\
256    & 1.350 & 6.286 & 12.277 \\
512    & 1.764 & 8.111 & 15.854 \\
1024   & 2.889 & 11.636 & 22.365 \\
\hline
\end{tabular}
\end{table}

Due to the low-rank phase space representation, the computational cost is significantly reduced. Algorithm \ref{algo: qrsvd} demands $\sim N_vr^2$ from the reduced QR factorizations and matrix multiplications, and $\sim r^3$ flops from the SVD. Looking at equations \eqref{eq: Cipmn_v2}-\eqref{eq: Cip_v2_fluxes}, Algorithm \ref{algo: qrsvd} is applied $4(k+1)+3$ times per node. Assuming $r\ll N_v$ and very small $k\in\{0,1,2\}$, the cost of evaluating the righthand side of equation \eqref{eq: Cipmn_v2} is dominated by $\sim N_vr^2$ flops. Under these same assumptions, the cost of solving the moment equations (per node) is also dominated by $\sim N_vr^2$ flops due to the \texttt{LRDI} function. Since we are working on $N_x(k+1)$ nodes, the overall computational complexity under the low-rank assumption is dominated by $\mathcal{O}(N_xN_vr^2)$ flops. Notably, the computational complexity is only linear with respect to $N_v$. Table \ref{table: runtimes} presents the CPU runtime under velocity mesh refinement, where we use spatial mesh $N_x=40$, $k\in\{0,1,2\}$, and final time $t_f=0.1$. For NDG1, NDG2, and NDG3, the observed complexity is actually sub-linear, although further refinement should theoretically approach linear complexity. The increase in CPU runtime from NDG1 to NDG2 to NDG3 is due to the growth in $k$ from Algorithm \ref{algo: qrsvd}, as well as working on $k+1$ nodes per element.

\begin{figure}[h!]
\begin{minipage}[b]{0.47\linewidth}
\caption{Absolute error of the total moments. NDG3, $N_x=N_v=100$, $\varepsilon=1$, truncation tolerance $\vartheta=10^{-6}$.}
\label{fig: smooth_momenterror_eps1_tolE-6}
\includegraphics[width=\textwidth]{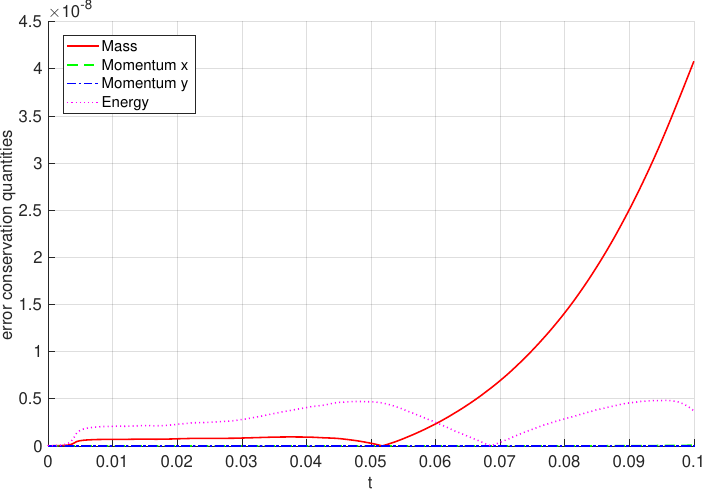}
\end{minipage}
\begin{minipage}[b]{0.05\linewidth}
\end{minipage}
\begin{minipage}[b]{0.47\linewidth}
\caption{Absolute error of the total moments. NDG3, $N_x=N_v=100$, $\varepsilon=10^{-6}$, truncation tolerance $\vartheta=10^{-6}$.}
\label{fig: smooth_momenterror_epsE-6_tolE-6}
\includegraphics[width=\textwidth]{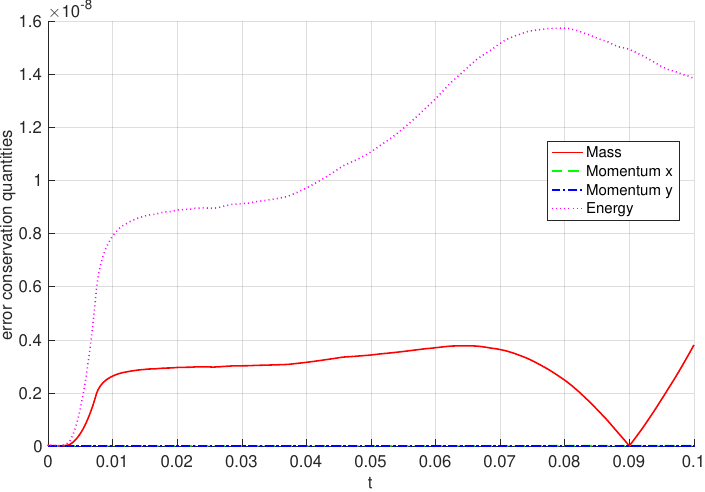}
\end{minipage}
\end{figure}

\begin{figure}[h!]
\begin{minipage}[b]{0.47\linewidth}
\caption{Absolute error of the total moments. NDG3, $N_x=N_v=100$, $\varepsilon=1$, truncation tolerance $\vartheta=10^{-15}$.}
\label{fig: smooth_momenterror_eps1_tolE-15}
\includegraphics[width=\textwidth]{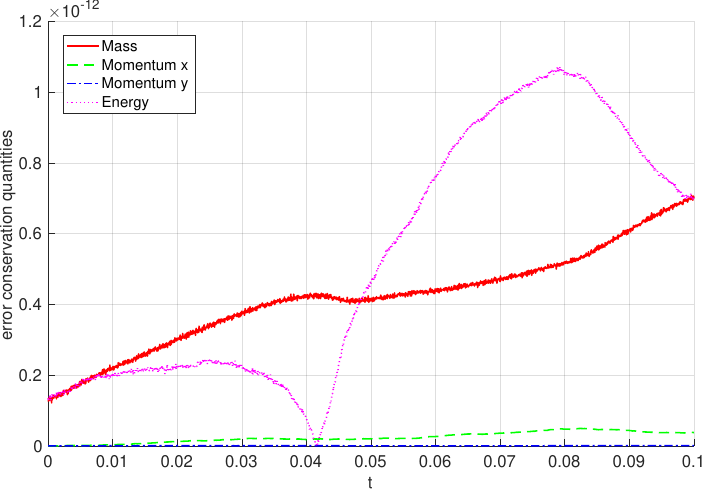}
\end{minipage}
\begin{minipage}[b]{0.05\linewidth}
\end{minipage}
\begin{minipage}[b]{0.47\linewidth}
\caption{Absolute error of the total moments. NDG3, $N_x=N_v=100$, $\varepsilon=10^{-6}$, truncation tolerance $\vartheta=10^{-15}$.}
\label{fig: smooth_momenterror_epsE-6_tolE-15}
\includegraphics[width=\textwidth]{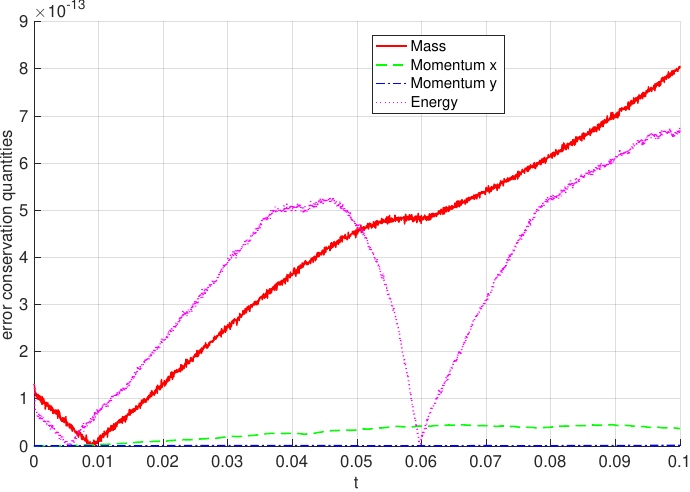}
\end{minipage}
\end{figure}

For this smooth problem, the total mass, momentum, and energy are  $\int_{\Omega_x}{\left\langle\big(1,v_x,v_y,(v_x^2+v_y^2)/2)\big)^Tf\right\rangle dx} = \int_{\Omega_x}{\mathbf{U}(x,t)dx} = (2\pi,0,0,2\pi)^T$. Since the QCM conserves the moments at the discrete level, the mass, momentum, and energy are conserved up to the truncation tolerance $\vartheta$. Figures \ref{fig: smooth_momenterror_eps1_tolE-6}-\ref{fig: smooth_momenterror_epsE-6_tolE-15} show how well the proposed method conserves the total moments by plotting the absolute errors using NDG3, $\varepsilon\in\{1,10^{-6}\}$, final time $t_f=0.1$, spatial mesh $N_x=100$, and velocity mesh $N_v=100$. Figures \ref{fig: smooth_momenterror_eps1_tolE-6}-\ref{fig: smooth_momenterror_epsE-6_tolE-6} set $\vartheta=10^{-6}$, and while momentum is well-conserved by setting $u_x=u_y=0$, mass and energy conservation is affected by the truncation tolerance. Figures \ref{fig: smooth_momenterror_eps1_tolE-15}-\ref{fig: smooth_momenterror_epsE-6_tolE-15} set $\vartheta=10^{-15}$, and the moments are conserved almost to machine precision for both $\varepsilon=1$ and $\varepsilon=10^{-6}$.

\subsection{Standing shock problem}
We now consider a standing shock problem similar to the Mach 5 steady-state shock example in \cite{taitano2018adaptive,taitano2021cpc}, using a uniform spatial mesh. The initial distribution is a Maxwellian distribution defined by the following Riemann data: upstream conditions $n_0=0.62963$, $u_{x,0}=1.63712$, $u_{y,0}=0$, $T_0=0.595588$, and downstream conditions $n_1=1$, $u_{x,1}=1.03078$, $u_{y,1}=0$, $T_1=1$. The initial profiles are smoothed using hyperbolic tangents,
\begin{equation}
    n(x,0)=a_n\tanh{[\xi(x-\pi)]}+b_n,\quad u_{x}(x,0)=a_{u_x}\tanh{[\xi(x-\pi)]}+b_{u_x},\quad T(x,0)=a_T\tanh{[\xi(x-\pi)]}+b_T,
\end{equation}
where $\xi=1/h_x$ is the smoothing factor, and $a_{\phi}=(\phi_1-\phi_0)/2$, $b_{\phi}=(\phi_1+\phi_0)/2$. We also set the Knudsen number to $\varepsilon=10^{-13}$, $V_{\text{max}}=12$, $\Omega_x=[0,2\pi]$, and mesh $N_x=N_v=100$. We compute the numerical solutions with NDG2 up to $t_f=4$, by which time the shock profile has corrected itself from the hyperbolic tangents. To demonstrate the effectiveness of the post-processing slope limiters, we compare the numerical results using no limiter, the minmod limiter \cite{cockburn1999discontinuous,xiong2015high}, and the WENO limiter \cite{zhong2013simple} with parameters $\gamma_0=\gamma_2=0.001$ and $\gamma_1=0.998$. As seen in Figure \ref{fig: moments_SS}, significant oscillations occur when no slope limiter is used. Whereas, the minmod and WENO limiters both control spurious oscillations, with the WENO limiter performing slightly better. Although not shown, the bulk velocity in the $v_y$ direction, $u_y(x,t)$, remained zero (up to machine precision) for all three cases, but recall that $u_y(x,t)=0$.

\begin{figure}[h!]
\caption{Standing shock problem using NDG2, $t_f=4$, $N_x=N_v=100$, $\varepsilon=10^{-13}$. Left column: number density $n(x,t)$. Middle column: bulk velocity $u_x(x,t)$. Right column: temperature $T(x,t)$. Top row: no limiter. Middle row: minmod limiter. Bottom row: WENO limiter.}
\label{fig: moments_SS}
\begin{minipage}[b]{0.32\linewidth}
    \includegraphics[width=\textwidth]{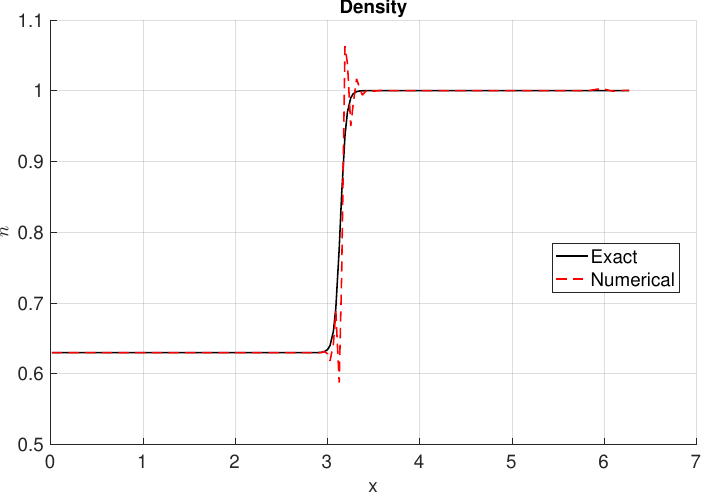}
    \includegraphics[width=\textwidth]{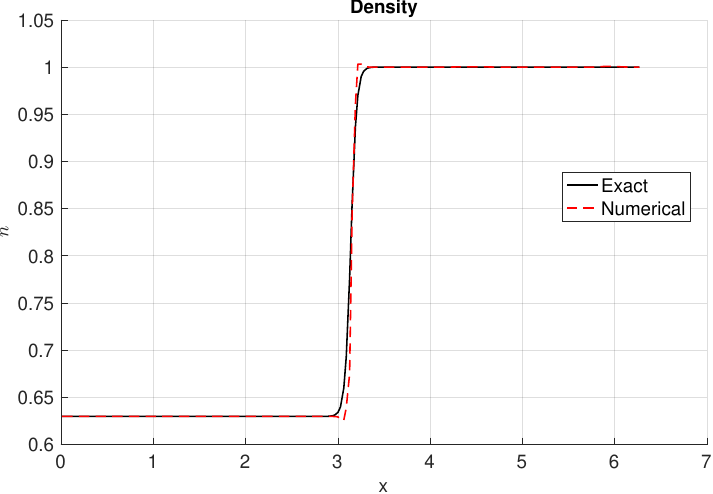}
    \includegraphics[width=\textwidth]{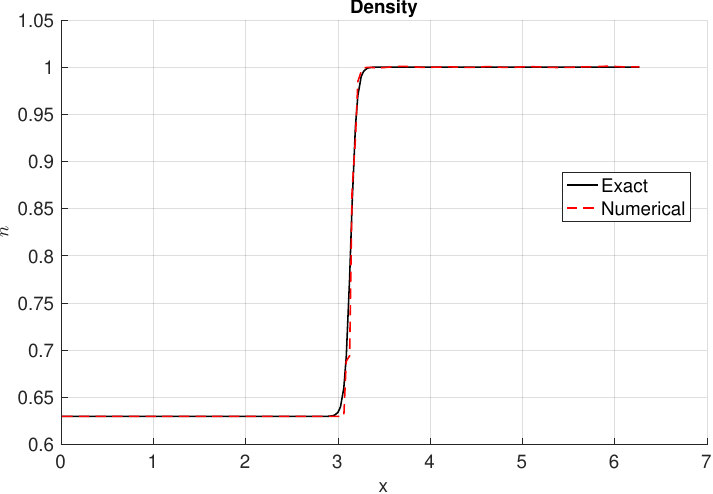}
\end{minipage}
\begin{minipage}[b]{0.32\linewidth}
    \includegraphics[width=\textwidth]{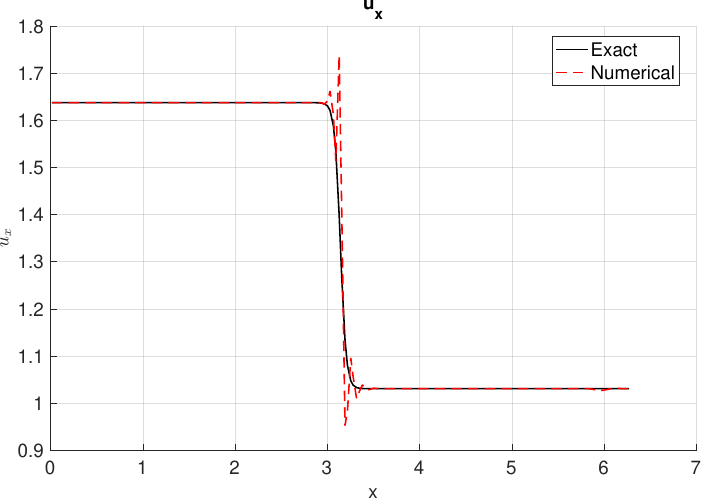}
    \includegraphics[width=\textwidth]{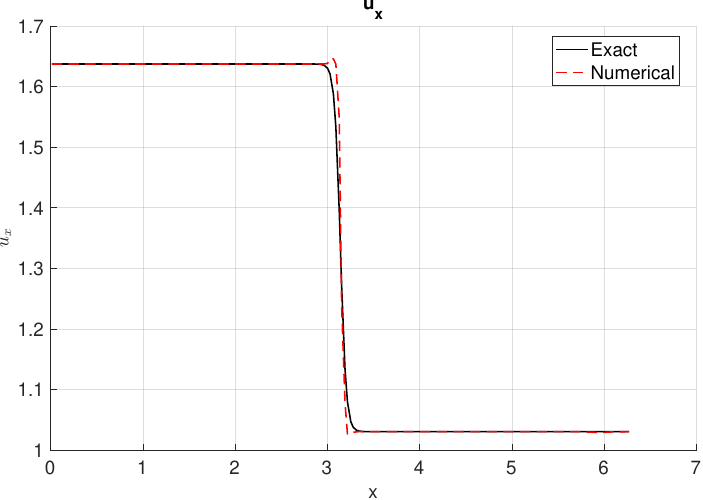}
    \includegraphics[width=\textwidth]{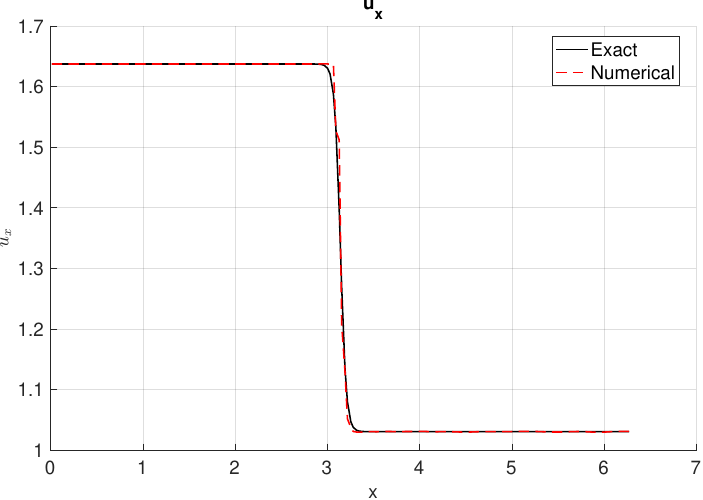}
\end{minipage}
\begin{minipage}[b]{0.32\linewidth}
    \includegraphics[width=\textwidth]{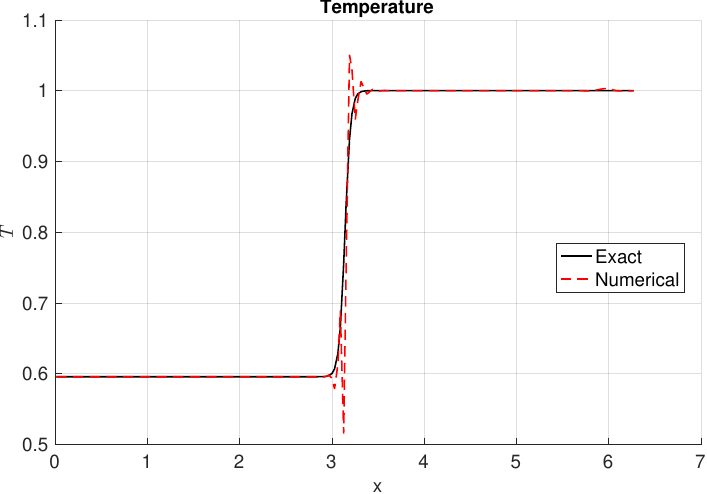}
    \includegraphics[width=\textwidth]{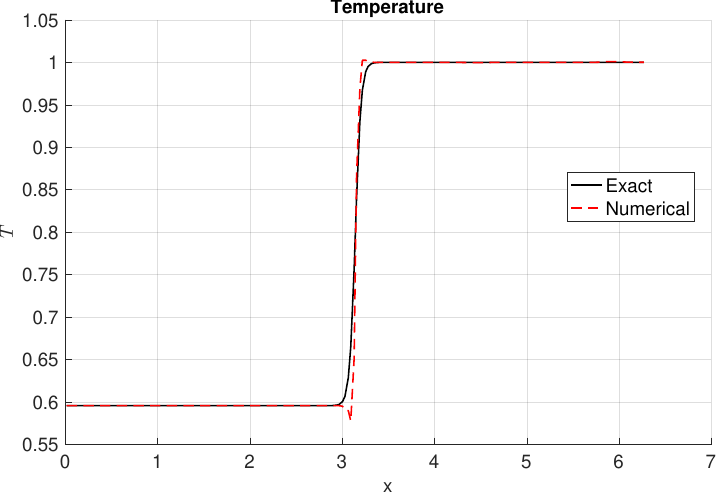}
    \includegraphics[width=\textwidth]{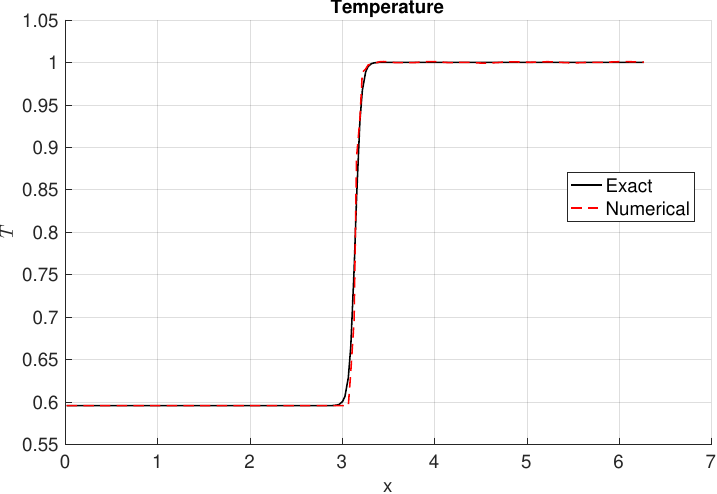}
\end{minipage}
\end{figure}

\subsection{1d2v Sod shock tube}

We consider the Sod shock tube problem in 1d2v over the domain $\Omega_x\times\Omega_v=[0,1]\times[-5,5]^2$, with an initial Maxwellian distribution defined by the following initial states:
\begin{equation}
    \Big(n(x,0),u_x(x,0),u_y(x,0),p(x,0)\Big) = 
    \begin{cases}
        (1,0,0,1),&0\leq x\leq 1/2,\\
        (1/8,0,0,0.1),&1/2\leq x\leq 1,
    \end{cases}
\end{equation}
where $p=nT$ is the pressure. The energy when dealing with an ideal polytropic gas is $E=\frac{n|\mathbf{u}|^2}{2}+\frac{p}{\gamma-1}$ from the constitutive equation relating the pressure and internal energy. This is consistent with the energy equation in \eqref{eq: U} since $\gamma=(d+2)/d$. In the following results, the numerical solutions are computed using NDG2, uniform meshes $N_x=150$ and $N_v=100$, and up to time $t_f=0.14$. As with the standing shock problem, we start by demonstrating the effectiveness of the post-processing slope limiters. As seen in Figure \ref{fig: moments_Sod}, oscillations appear when no limiter is used -particularly for the bulk velocity and temperature-, but the minmod and WENO (with $\gamma_0=\gamma_2=0.001$ and $\gamma_1=0.998$) limiters both do a comparable job controlling the oscillations. We also plot the moments for varying Knudsen numbers $\varepsilon\in\{1,10^{-1},10^{-2},10^{-3},10^{-6},10^{-13}\}$ in Figure \ref{fig: moments_Sod_varyKnudsen} using the WENO limiter; the results using the minmod limiter are similar. The results are consistent with those shown in \cite{bennoune2008uniformly,degond2005smooth,xiong2015high}, although our numerical solution for $\varepsilon=10^{-13}$ has slight oscillatory behavior near the discontinuity in the temperature. In all of our Sod shock tube simulations, the bulk velocity in the $v_y$ direction, $u_y(x,t)$, remained zero (up to machine precision) since we set $u_y(x,t)=0$.

\begin{figure}[h!]
\caption{Sod shock tube problem using NDG2, $t_f=0.14$, $N_x=150$, $N_v=100$, $\varepsilon=10^{-13}$. Left column: number density $n(x,t)$. Middle column: bulk velocity $u_x(x,t)$. Right column: temperature $T(x,t)$. Top row: no limiter. Middle row: minmod limiter. Bottom row: WENO limiter.}
\label{fig: moments_Sod}
\begin{minipage}[b]{0.32\linewidth}
    \includegraphics[width=\textwidth]{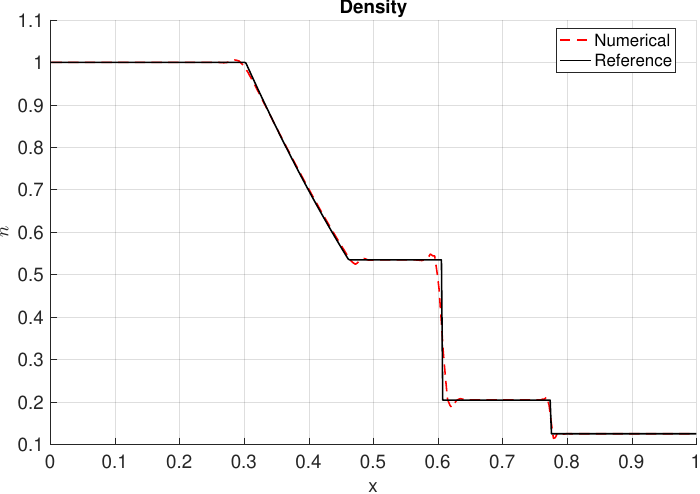}
    \includegraphics[width=\textwidth]{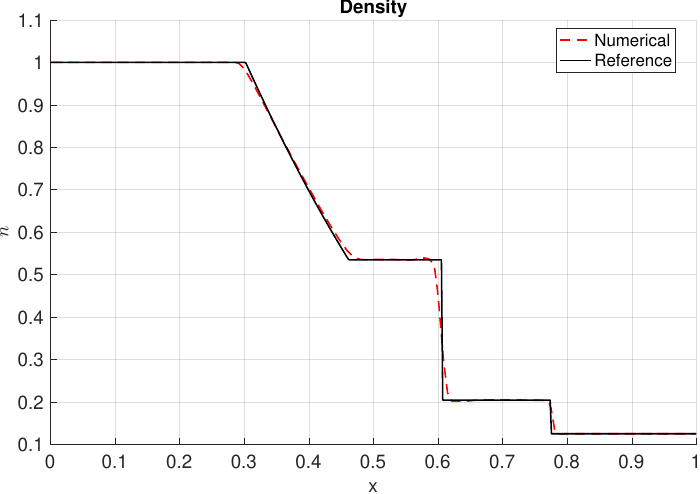}
    \includegraphics[width=\textwidth]{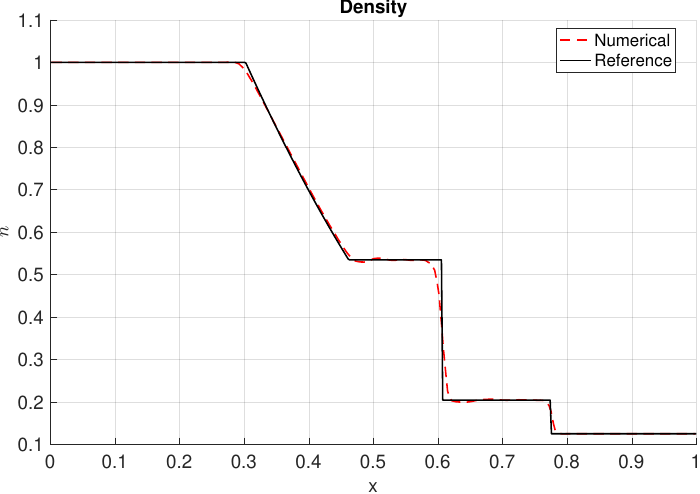}
\end{minipage}
\begin{minipage}[b]{0.32\linewidth}
    \includegraphics[width=\textwidth]{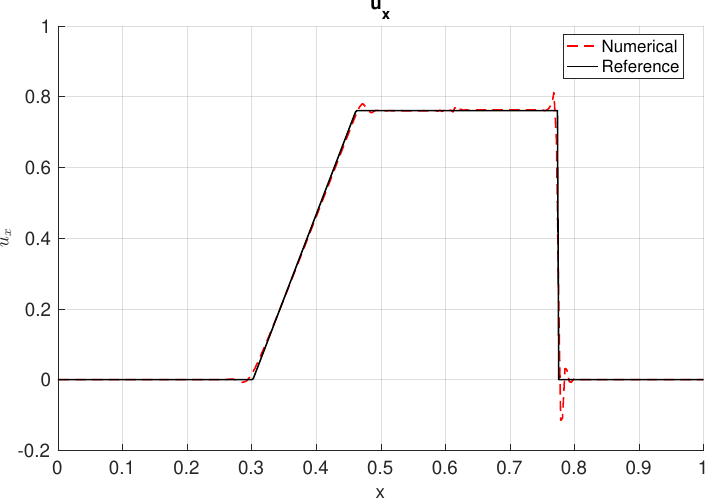}
    \includegraphics[width=\textwidth]{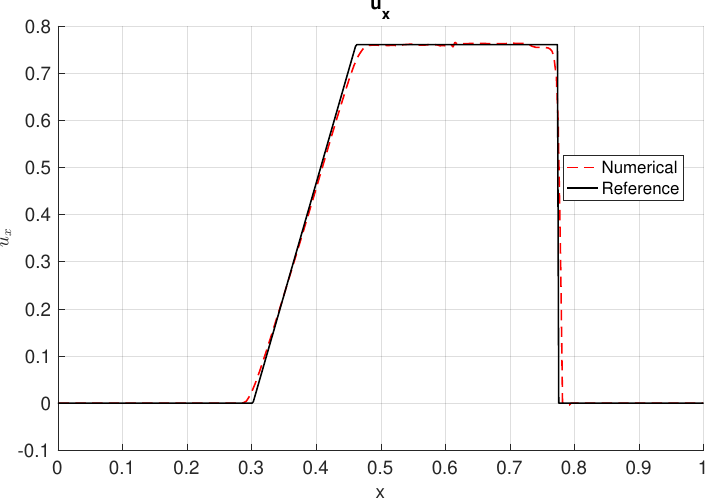}
    \includegraphics[width=\textwidth]{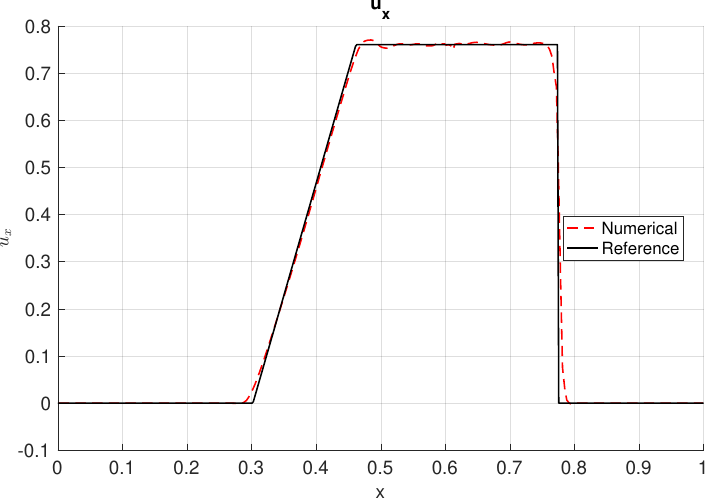}
\end{minipage}
\begin{minipage}[b]{0.32\linewidth}
    \includegraphics[width=\textwidth]{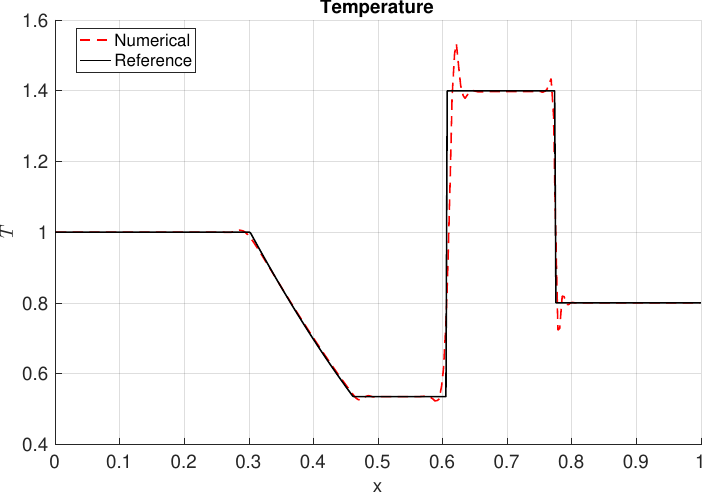}
    \includegraphics[width=\textwidth]{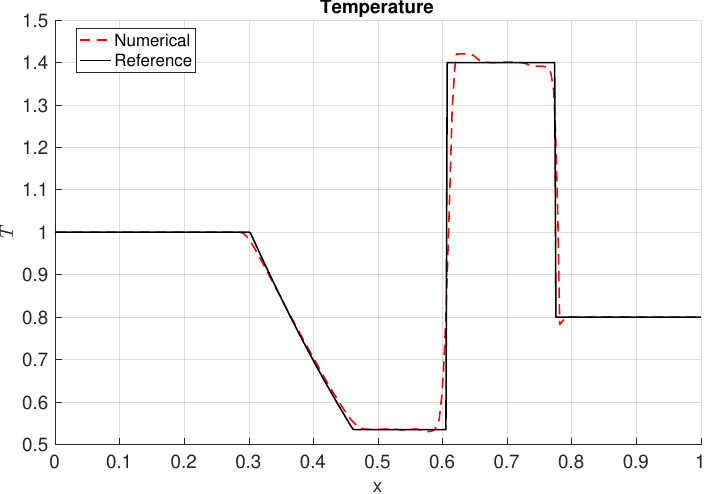}
    \includegraphics[width=\textwidth]{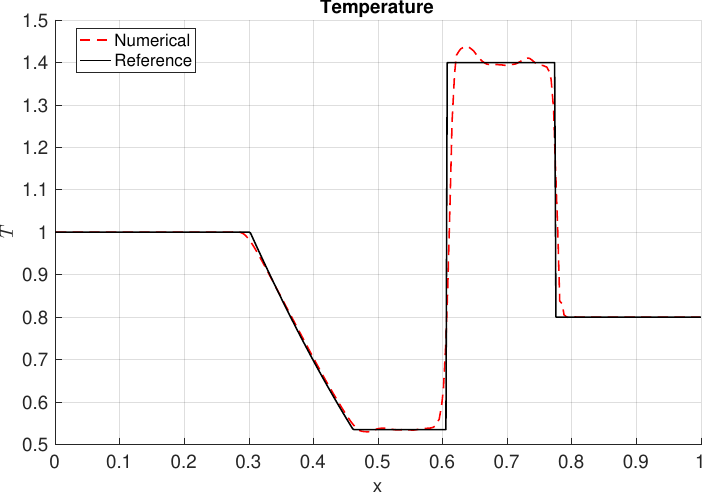}
\end{minipage}
\end{figure}

The main advantage of the proposed method over full-rank nodal DG methods is the reduced computational cost resulting from the low-rank methodology. We discussed and verified the computational complexity of the method for the smooth problem \eqref{eq: smoothIC}. For the Sod shock tube problem, we also show the rank of the solution (in velocity space) at all the nodal points in physical space. We use NDG2 and WENO limiter with parameters $\gamma_0=\gamma_2=0.001$ and $\gamma_1=0.998$, fix the spatial mesh $N_x=150$, and run the solution up to time $t_f=0.14$. Figure \ref{fig: ranks_Sod_NDG2} shows that for moderate Knudsen numbers $\varepsilon\in\{10^{-3},10^{-6}\}$, the rank slightly increased near the locations of the discontinuities where the solution transitions from one state to another; this also occurs if the spatial mesh is too coarse. Although we do not include the plot, the solution remained uniformly rank-1 when $\varepsilon=10^{-13}$, demonstrating the asymptotic preservation of the kinetic solution in the limit $\varepsilon\rightarrow 0$, in which case we know the solution must be a rank-1 Maxwellian distribution at all spatial nodes.

\begin{figure}[h!]
\caption{Sod shock tube problem using NDG2, WENO limiter, $t_f=0.14$, $N_x=150$, $N_v=100$, and varying Knudsen number $\varepsilon\in\{1,10^{-1},10^{-2},10^{-3},10^{-6},10^{-13}\}$. Left: number density $n(x,t)$. Middle: bulk velocity $u_x(x,t)$. Right: temperature $T(x,t)$.}
\label{fig: moments_Sod_varyKnudsen}
\begin{minipage}[b]{0.32\linewidth}
    \includegraphics[width=\textwidth]{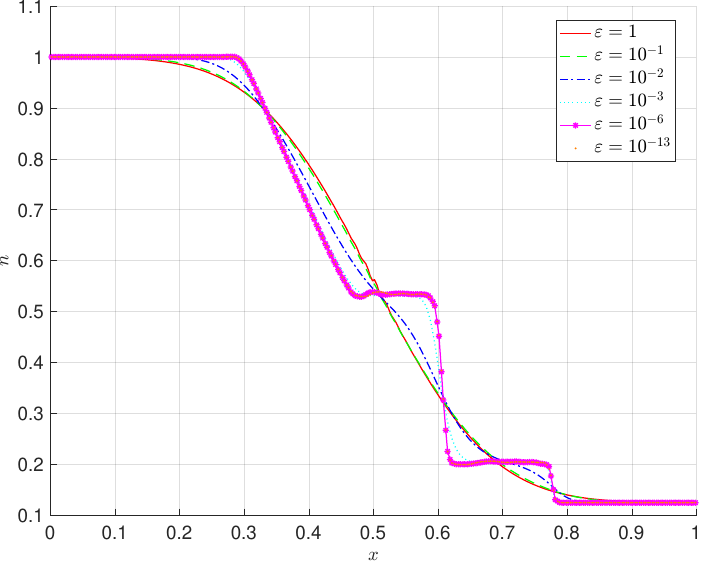}
\end{minipage}
\begin{minipage}[b]{0.32\linewidth}
    \includegraphics[width=\textwidth]{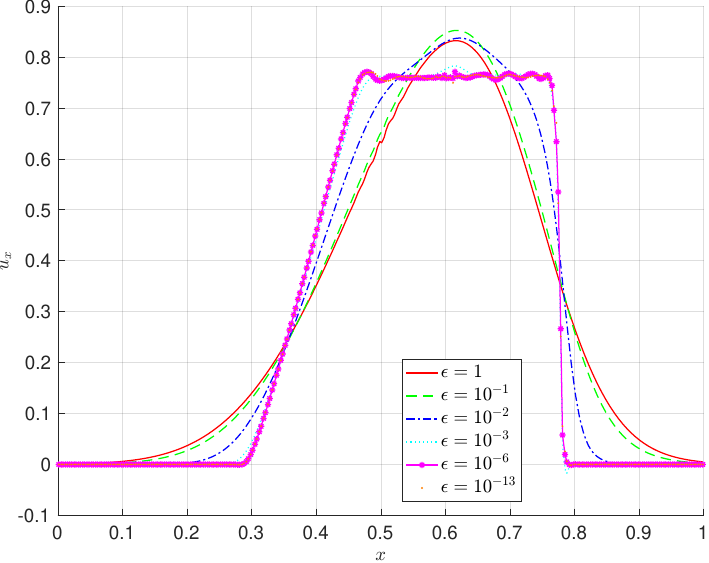}
\end{minipage}
\begin{minipage}[b]{0.32\linewidth}
    \includegraphics[width=\textwidth]{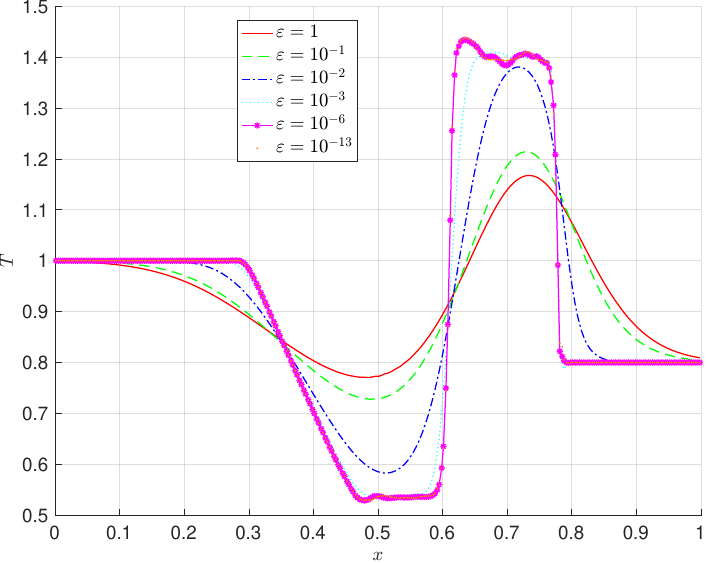}
\end{minipage}
\end{figure}

\begin{figure}[h!]
\caption{Sod shock tube problem using NDG2, WENO limiter, $t_f=0.14$, $N_x=150$, and varying $N_v=\{32,64,128,256,512\}$. The rank in velocity space is plotted at all nodal points in space. Left: $\varepsilon=10^{-3}$. Right: $\varepsilon=10^{-6}$.} 
\label{fig: ranks_Sod_NDG2}
\begin{minipage}[b]{0.47\linewidth}
    \includegraphics[width=\textwidth]{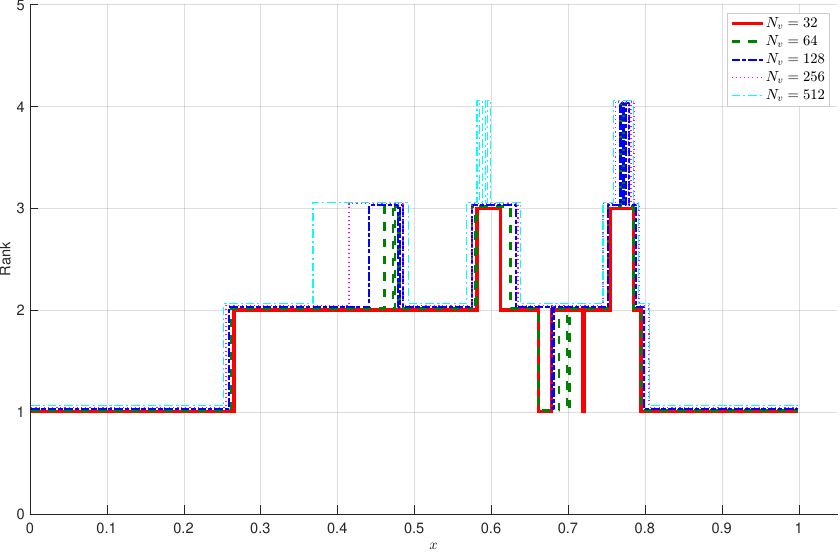}
\end{minipage}
\begin{minipage}[b]{0.05\linewidth}
\end{minipage}
\begin{minipage}[b]{0.47\linewidth}
    \includegraphics[width=\textwidth]{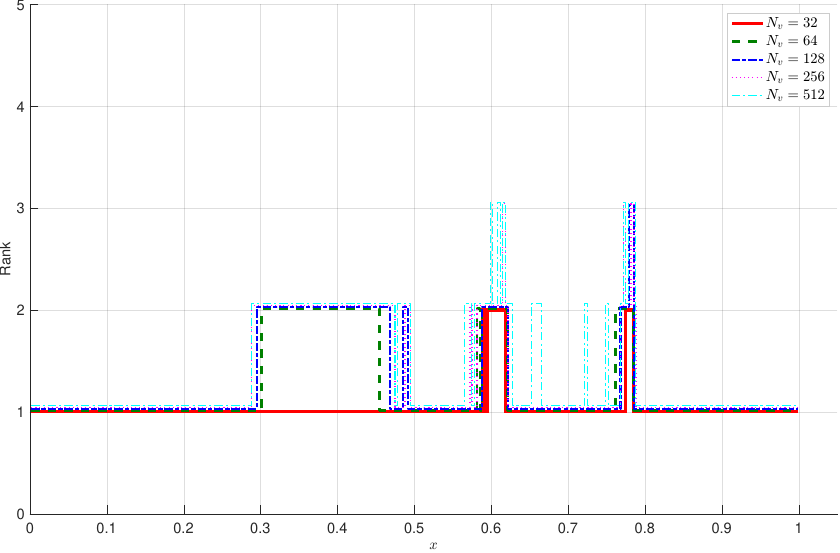}
\end{minipage}
\end{figure}

\subsection{Variable $\varepsilon(x)$ problem}

Following the example presented in \cite{filbet2010class}, we consider a spatially dependent Knudsen number $\varepsilon(x)>0$,
\begin{equation}
    \varepsilon(x)=\varepsilon_0+\frac{1}{2}[\tanh(1-11x)+\tanh(1+11x)],
    \label{eq:var_epsilon}
\end{equation}
and an initial distribution function far from Maxwellian,
\begin{equation}
    f_0(x,\mathbf{v})=\frac{n_0}{2}\left[\exp\left(-\frac{|\mathbf{v}-\mathbf{u}_0|^2}{T_0}\right)+\exp\left(-\frac{|\mathbf{v}+\mathbf{u}_0|^2}{T_0}\right)\right],
    \label{eq:initial_data_variable}
\end{equation}
with $\mathbf{u}_0=(0.75,-0.75)$, $n_0(x)=1+0.5\sin(\omega x)$, and $T_0(x)=0.25+0.1\cos(\omega x)$. The spatial domain $\Omega_x=[-L,L]$ with $L=0.5$ and $\omega=\pi/L$. Periodic boundary conditions are used in physical space $x$, and the velocity domain is taken to be \mbox{$\Omega_v=[-10,10]^2$}.

In Figure \ref{fig: moments_variable_eps}, we show the density $n$, mean velocity $u_x$, and temperature $T$ at times $t=0.25,\,0.5,\,0.75$. We used $\varepsilon_0=10^{-3}$, uniform mesh $N_x=200$ and $N_v=100$, tolerance $\vartheta=10^{-6}$, and IMEX222. A reference solution is also computed by solving with NDG1 using a much finer spatial mesh $N_x=1000$, along with $N_v=100$ and IMEX222. The NDG1 simulation over the coarser mesh had the most numerical diffusion, whereas NDG2 and NDG3 sharply captured the steep gradients and discontinuities. Notably, our results match those presented in \cite{filbet2010class}.

\begin{figure}[h!]
\caption{Mixed regime problem with spatial dependent Knudsen number \eqref{eq:var_epsilon} with $\varepsilon_0=10^{-3}$. $N_x=100$, $N_v=100$. From left to right: time $t=0.25,
\,0.5,\,0.75$. From top to bottom: density $n$,  mean velocity $u_x$, temperature $T$. The WENO limiter is applied. Reference solution obtain with NDG1 with $N_x=1000$, $N_v=100$ and IMEX222.}
\label{fig: moments_variable_eps}
\begin{minipage}[b]{0.32\linewidth}
    \includegraphics[width=\textwidth]{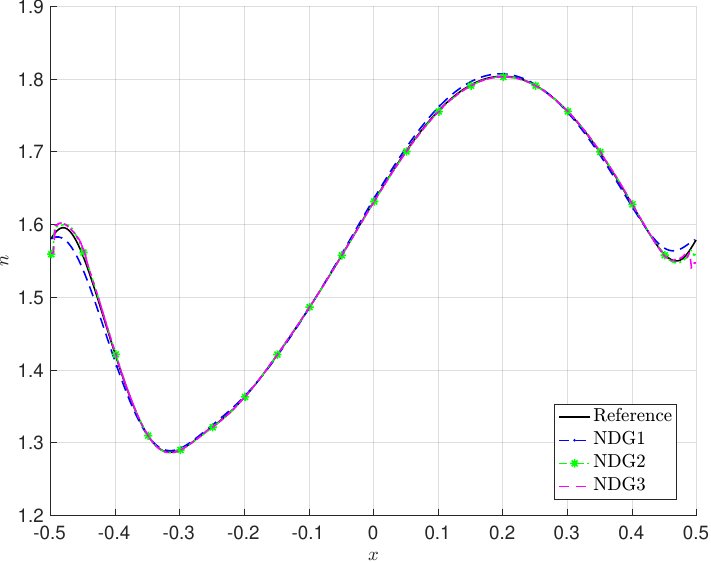}
    \includegraphics[width=\textwidth]{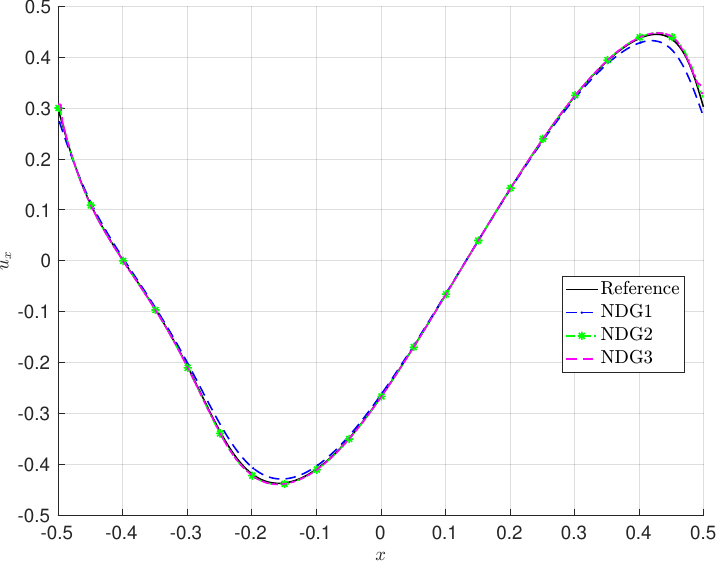}
     \includegraphics[width=\textwidth]{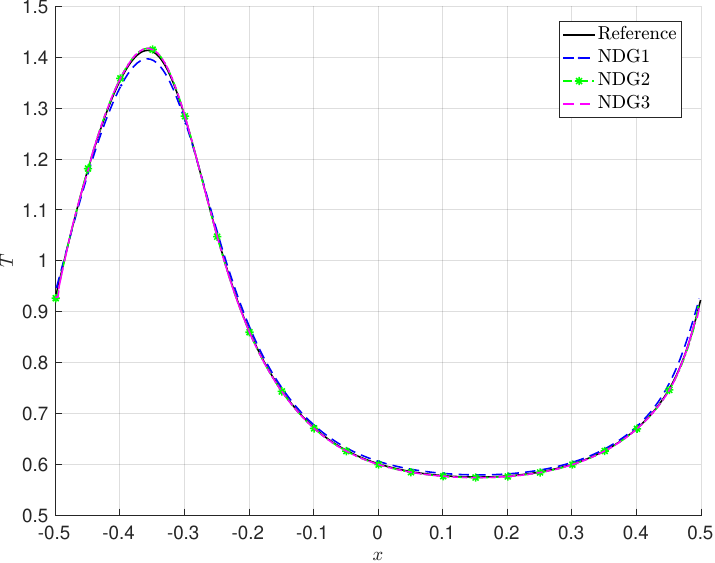}
\end{minipage}
\begin{minipage}[b]{0.32\linewidth}
    \includegraphics[width=\textwidth]{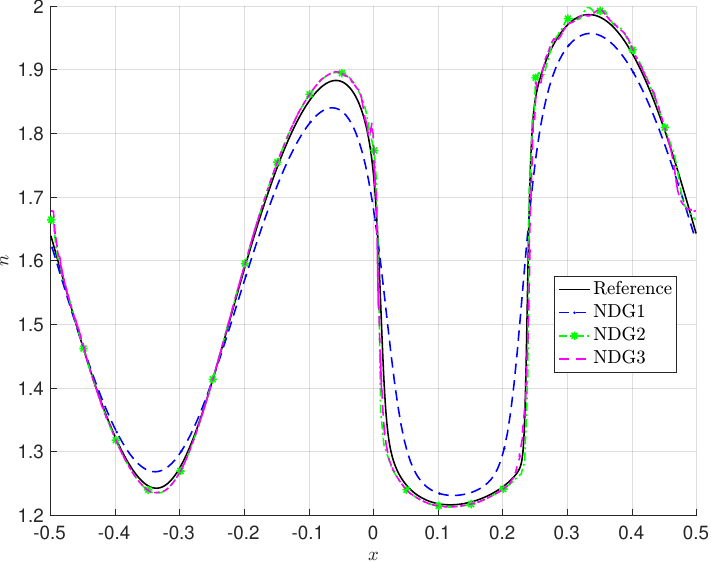}
    \includegraphics[width=\textwidth]{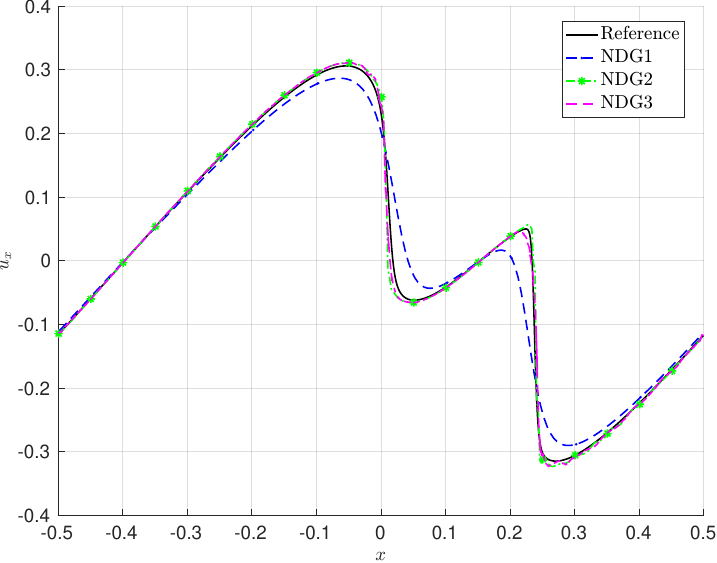}
     \includegraphics[width=\textwidth]{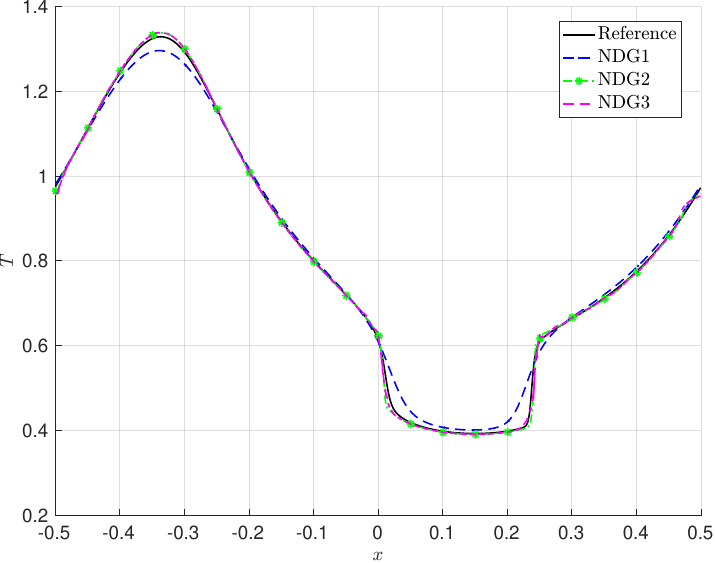}
\end{minipage}
\begin{minipage}[b]{0.32\linewidth}
    \includegraphics[width=\textwidth]{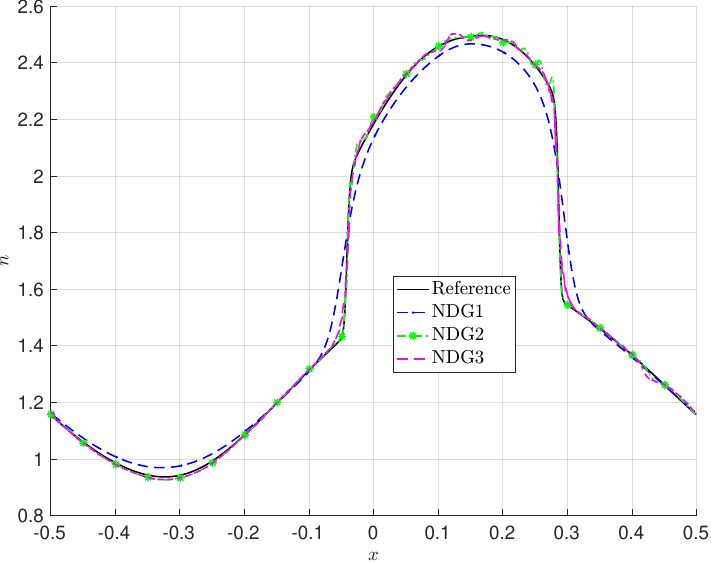}
    \includegraphics[width=\textwidth]{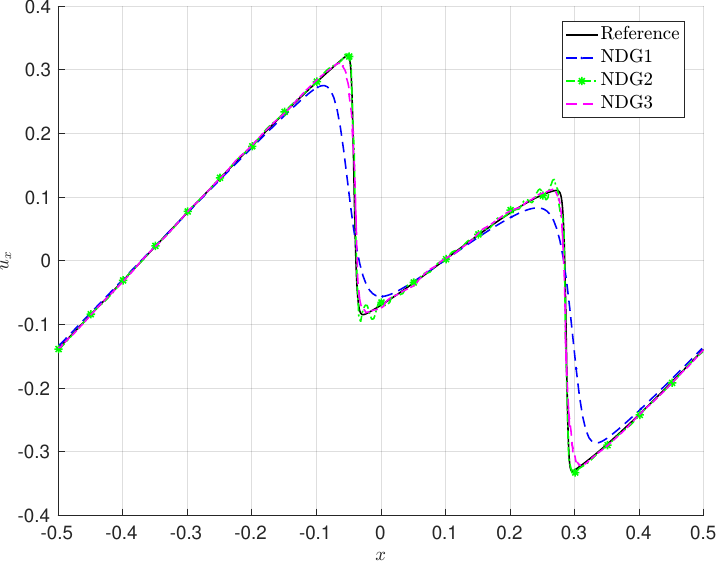}
     \includegraphics[width=\textwidth]{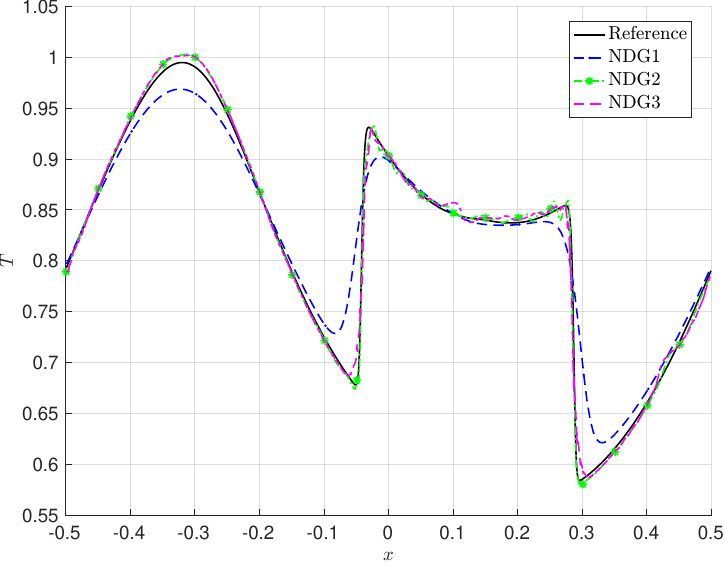}
\end{minipage}

\end{figure}

\section{Conclusion}\label{sec: conclusion}

In this paper, we proposed a nodal DG-IMEX method with low-rank velocity space representation for solving the BGK equation. In doing so, we leverage the high-order accuracy from the nodal DG formulation in physical space, as well as the computational efficiency that comes from working with a low-rank decomposition in velocity space. We further couple this with a second-order IMEX RK method for overall second-order accuracy, although extending to higher-order IMEX RK methods is straightforward. A post-processing step applies slope limiters to the moments at each RK stage to control spurious oscillations. We also show that the numerical solution converges to the corrected Maxwellian distribution in the asymptotic limit as $\varepsilon\rightarrow 0$. Numerical tests verify our method's computational complexity, order of accuracy, and robustness. The current paper works with the 1d2v BGK model, but the proposed method offers a promising approach to break the curse of dimensionality as we extend to higher-dimensional phase space in future work. The following directions will be subject to our future algorithmic and numerical investigations: higher dimensions in phase space with unstructured meshes and slope limiters, rigorous AP property with velocity discretization and compression,and positivity preservation for low-rank solutions.

\subsection{Data Availability} The source code generated and analyzed in the current study is available at \texttt{github.com/afgalindo/1D2VBGK}.

\subsection{Competing interests}
The authors declare that they have no competing interests.


\bibliographystyle{plain}
\bibliography{refs}
\end{document}